\documentclass[12pt]{amsart}
\usepackage{amscd}
\usepackage{latexsym}
\usepackage{amssymb}

\usepackage{ytableau}
\usepackage{hyperref}

\allowdisplaybreaks

\newtheorem{theorem}{Theorem}[section]
\newtheorem{proposition}[theorem]{Proposition}
\newtheorem{prop}[theorem]{Proposition}
\newtheorem{lemma}[theorem]{Lemma}
\newtheorem{notation}[theorem]{Notation}
\newtheorem{definition}[theorem]{Definition}

\newtheorem{corollary}[theorem]{Corollary}
\newtheorem{remark}[theorem]{Remark}

\newtheorem{example}[theorem]{Example}

\begin{document}

%
%

\title[Thick representations of the symmetric group]{\small The classification of thick representations of the symmetric group}

\author{Kazunori NAKAMOTO, Shingo OKUYAMA, and Yasuhiro OMODA}
\address{Center for Medical Education and Sciences, Faculty of Medicine, 
University of Yamanashi}
\email{nakamoto@yamanashi.ac.jp}
\address{
National Institute of Technology, 
Kagawa College}
\email{okuyama@di.kagawa-nct.ac.jp}
\address{
National Institute of Technology, 
Akashi College}
\email{omoda@akashi.ac.jp}

\thanks{The authors were partially supported by JSPS KAKENHI Grant Numbers JP20K03509 and JP24K06686. } 

\subjclass[2020]{Primary 20C30; Secondary 05E10.}

\keywords{Thick representation, Dense representation, Symmetric group.} 

\date{20 March 2026 (ver. 1.0.0)}

\begin{abstract} 
We give the classification of thick representations and dense representations of the symmetric group over a field of characteristic zero.   
\end{abstract}

%
%

\maketitle

%
%

\section{Introduction}\label{section:Introduction} 

Let $V$ be a  finite-dimensional vector space over a field $k$. 
For a group representation $\rho : G \to {\rm GL}(V)$, we say that $\rho$ is 
{\it thick} if for any subspaces $V_1, V_2$ of $V$ with $\dim_k V_1+\dim_k V_2=\dim_k V$, 
there exists $g\in G$ such that $(\rho(g)V_1)\cap V_2=0$. 
We also say that $\rho$ is {\it dense} if the exterior representation 
$\bigwedge^m\rho : G \to {\rm GL}(\bigwedge^m V)$ is irreducible for any $0 < m < \dim_k V$ (Definition~\ref{def:thickanddense}).  
If $\rho$ is dense, then it is thick; and if it is thick, then it is irreducible (Proposition~\ref{prop:densethickirreducible}).

Regarding the classification problem of thick representations, 
in the case of finite-dimensional irreducible representations over ${\Bbb C}$ of connected 
complex semi-simple Lie groups, the notion of 
thickness fits well within the framework of weight theory and the classification has been successfully achieved (Theorems~\ref{th:WMF}, \ref{th:listofthick}, \ref{th:listofdense}, and \cite[Theorem~3.12]{NO2}).  
In this paper, we discuss the classification of thick representations and dense representations of the symmetric group over a field of characteristic zero, a prototypical example among finite groups.     

Let $S_n$ be the symmetric group of degree $n$. 
For a partition  $\lambda$ of $n$, we denote by $V_{\lambda}$ the irreducible representation of $S_n$ 
over a field of characteristic zero corresponding to $\lambda$.  
The main theorems of this paper are the following: 

\begin{theorem}
\label{th:main1}
For a finite-dimensional irreducible 
representation $\rho$ of $S_n$ over a field of characteristic zero, 
$\rho$ is thick if and only if it is dense. 
\end{theorem}

\begin{theorem}
\label{th:main2}
Let ${\Bbb K}$ be a field of characteristic zero. 
The thick (or dense) representations of the symmetric group $S_n$ over ${\Bbb K}$ are those on the following list: 
\begin{enumerate}
\item\label{item:Snrep1} the trivial representation $V_{(n)}$ of $S_n$ for $n\ge 1$, 
\item\label{item:Snrep2} the sign representation $V_{(1^n)}$ of $S_n$ for $n\ge 2$, 
\item\label{item:Snrep3} the standard representation $V_{(n-1, 1)}$ of $S_n$ for $n \ge 3$, 
\item\label{item:Snrep4} the product of the standard and sign representation $V_{(2, 1^{n-2})}$ of $S_n$ for $n \ge 4$, 
\item\label{item:Snrep5} the $2$-dimensional irreducible representation $V_{(2^2)}$ of $S_4$,  
\item\label{item:Snrep6} the $5$-dimensional irreducible representations $V_{(2^3)}$ and $V_{(3^2)}$ of $S_6$.     
\end{enumerate} 
\end{theorem}

In \cite{NO5}, we announced the main theorems, and the present paper is the detailed version of \cite{NO5}. 

\bigskip 

The classification of thick (or dense) representations can be extended in future work to the classification of $(i, j)$-thick (or $(i, j)$-dense) representations of the symmetric group 
(for the definitions of $(i, j)$-thickness and $(i, j)$-denseness, 
see Definitions~\ref{def:ijthick} and~\ref{def:ijdense}). 
This problem reduces to 
determining the thick diagram $T(\rho)$ and the dense diagram $D(\rho)$ 
for each irreducible representation $\rho$ of the symmetric group (see Definitions~\ref{def:thickdiagram}). 
However, a detailed investigation is beyond the scope of the present paper. Instead, we focus on the classification of thick representations 
and dense representations of the symmetric group, and merely list several hints toward a future determination of the thick diagram $T(\rho)$ for each irreducible representation $\rho$ of the symmetric group; a more refined analysis is left for future work.

Throughout this paper, ${\Bbb K}$ denotes a field of characteristic zero. For a subset $S$ of a vector space $V$ over ${\Bbb K}$, we denote by $\langle S \rangle$ 
the subspace of $V$ spanned by $S$. 
For a partition $\lambda$ of a positive integer $n$, 
let ${}^t \lambda$ denote the conjugate partition to $\lambda$.   
In other words, ${}^t \lambda$ is defined by interchanging the rows and columns of $\lambda$ as a Young diagram.  

The organization of this paper is as follows: in Section~\ref{section:Preliminaries}, we review thick representations and dense representations. We also introduce the  classification of thick representations and dense representations over ${\Bbb C}$ of connected complex simple Lie groups.  In Section~\ref{section:Specht}, we review Specht modules in order to treat irreducible representations of the symmetric group concretely. 
In Section~\ref{section:denserepresentations}, we classify dense representations of the symmetric group.   
In Section~\ref{section:thickrepresentations}, we give an outline of the proof of 
Theorems~\ref{th:main1} and \ref{th:main2}.   
In Sections~\ref{section:s=2}--\ref{section:sge4}, 
we divide the proof into three cases. 
More precisely, Section~\ref{section:s=2} is devoted to the case $\lambda=(m_1, m_2)$, Section~\ref{section:s=3} to the case $\lambda=(m_1, m_2, m_3)$, and Section~\ref{section:sge4} to the case $\lambda=(m_1, \ldots, m_s)$ with $s\ge 4$.

\section{Preliminaries}\label{section:Preliminaries}  

In this section, we review thick representations and dense representations. 
For details, see \cite{NO1} and \cite{NO2}. We also introduce the  classification of thick representations and dense representations over ${\Bbb C}$ of connected complex simple Lie groups.  Throughout this section, we assume that $\rho : G \to {\rm GL}(V)$ is an 
$n$-dimensional representation of a group $G$ over a field $k$ unless otherwise stated.   


\begin{definition}[{\cite[Definitions~2.1 and 2.3]{NO1} and \cite[Definitions~2.1 and 2.2]{NO2}}]\label{def:thickanddense}\rm 
Let $\rho : G \to {\rm GL}(V)$ be an $n$-dimensional representation of a group $G$ over a field $k$. 
We say that $\rho$ is {\it $m$-thick} if for any subspaces $V_1, V_2$ of $V$ with 
$\dim_{k} V_1=m$ and $\dim_{k}V_2 = n - m$, there exists $g \in G$ such that $(\rho(g)V_1)\cap V_2= 0$. 
If $\rho$ is $m$-thick for any $0 < m < n$, then we say that $\rho$ is {\it thick}. 
We also say that $\rho$ is {\it $m$-dense} if the exterior representation 
$\bigwedge^m\rho : G \to {\rm GL}(\bigwedge^m V)$ is irreducible. 
If $\rho$ is $m$-dense for any $0 < m < n$, then we say that $\rho$ is {\it dense}. 
\end{definition}

\begin{proposition}[{\cite[Proposition~2.6]{NO1} and 
\cite[Proposition~2.7]{NO2}}]\label{prop:mthickn-mthick} 
For an $n$-dimensional representation $\rho : G \to {\rm GL}(V)$, we have 
\begin{align*}
m\mbox{-thick}  & \Longleftrightarrow (n-m)\mbox{-thick}, \\
m\mbox{-dense} & \Longleftrightarrow (n-m)\mbox{-dense}. 
\end{align*}
\end{proposition}

\begin{proposition}[{\cite[Proposition~2.7 and Corollary~2.8]{NO1} and \cite[Proposition~2.8 and Corollary~2.9]{NO2}}]\label{prop:densethickirreducible} 
For an $n$-dimensional representation $\rho : G \to {\rm GL}(V)$ and $0<m<n$, we have 
\begin{eqnarray*} 
\begin{array}{ccccc} 
\mbox{$m$-dense} & \Longrightarrow & \mbox{$m$-thick}  & &  \\ 
 & & \Downarrow & & \\ 
\mbox{$1$-dense} & \Longleftrightarrow & \mbox{$1$-thick} & \Longleftrightarrow & \mbox{irreducible}. 
\end{array}  
\end{eqnarray*}
In particular, 
\[
dense \Longrightarrow thick \Longrightarrow \mbox{irreducible}. 
\]
\end{proposition}
 
\begin{corollary}[{\cite[Corollary~2.9]{NO1} and \cite[Corollary~2.10]{NO2}}]\label{cor:whendimle3}
For $n\le 3$, we have 
\[
dense \Longleftrightarrow thick \Longleftrightarrow \mbox{irreducible}.   
\]
\end{corollary}

\begin{notation}\rm
For a representation $\rho : G  \to {\rm GL}(V)$ of a group $G$, we sometimes call
$V$ itself a representation of $G$. For a group homomorphism $\varphi : G' \to G$, we denote by 
$\varphi^{\ast}V$ the representation $\rho\circ \varphi : G' \to {\rm GL}(V)$ of $G'$.  
\end{notation}

\bigskip

Let $\rho : G \to {\rm GL}(V)$ be an $n$-dimensional representation of a group $G$ over a field $k$. 
For positive integers $i$ and $j$ with $i+ j= n$, let us consider the $G$-equivariant perfect 
pairing $\bigwedge^i V \otimes_{k} \bigwedge^j V \stackrel{\wedge}{\longrightarrow} \bigwedge^{n} V \cong k$. 
For a $G$-invariant subspace $W$ of $\bigwedge^i V$, put $W^{\perp} := \{ y \in 
\bigwedge^j V 
\mid  x\wedge y = 0 \mbox{ for any } x \in W \}$. 
Then $W^{\perp}$ is also $G$-invariant. 
In particular, $\bigwedge^i V$ is irreducible if and only if so is $\bigwedge^j V$.

\begin{definition}[{\cite[Definition~2.10]{NO1} and \cite[Definition~2.11]{NO2}}]\label{def:vector}\rm
Let $V$ be an $n$-dimensional vector space over a field $k$. 
For a $d$-dimensional subspace $V'$ of $V$ with $0 < d < n$, we can consider a point  
$[ \bigwedge^d V' ]$ in the projective space ${\mathbb P}(\bigwedge^d V)$. 
In the sequel, we identify $[ \bigwedge^d V' ]$ with   
a non-zero vector $\bigwedge^d V' \in \bigwedge^d V$ (which is determined 
by $[ \bigwedge^d V' ]$ up to scalar) for simplicity.  
For a vector subspace $W \subseteq \bigwedge^d V$, we say that $W$ is 
{\it realizable} if $W$ contains a non-zero vector $\bigwedge^d V'$ 
obtained by a $d$-dimensional subspace $V'$ of $V$. 
\end{definition}

\begin{proposition}[{\cite[Proposition~2.11]{NO1} and \cite[Proposition~2.12]{NO2}}]\label{prop:condofm-thick}
Let $V$ be an $n$-dimensional representation of a group $G$. 
For $0 < m <n$,      
$V$ is not $m$-thick if and only if there exist 
$G$-invariant realizable subspaces $W_1 \subseteq \bigwedge^m V$ and $W_2 
\subseteq \bigwedge^{n-m} V$ 
such that $W_1^{\perp} = W_2$. 
\end{proposition}

\begin{definition}[{\cite{NO4}}]\rm\label{def:ijthick} 
Let $0 \le i, j \le n$. 
For an $n$-dimensional representation $\rho : G \to {\rm GL}(V)$, we say that $\rho$ is {\it $(i, j)$-thick} if for any subspaces $V_1, V_2$ of $V$ with 
$\dim_{k} V_1=i$ and $\dim_{k}V_2 =j$, 
there exists $g\in G$ such that $(\rho(g)V_1)\cap V_2= 0$. 
Note that if $\rho$ is $(i, j)$-thick, then $0 \le i+j \le n$.  
For $0<m<n$, $\rho$ is $m$-thick if and only if it is $(m, n-m)$-thick.  
We also see that $\rho$ is $(i, j)$-thick if and only if it is $(j, i)$-thick. 
\end{definition}

\begin{proposition}[{\cite{NO4}}]\label{prop:i+1jthick}
If an $n$-dimensional representation 
$\rho: G \to {\rm GL}(V)$ is $(i+1, j)$-thick, then $i+j<n$ and 
$\rho$ is $(i, j)$-thick. 
\end{proposition}

\proof
The statement directly follows from the definition of $(i, j)$-thickness. 
\qed 

\bigskip 

The following proposition will be used in Sections~\ref{section:thickrepresentations}--\ref{section:sge4}. 

\begin{proposition}[{\cite{NO4}}]\label{prop:ijthick} 
Let $i$ and $j$ be positive integers with $i+j \le n$. 
Assume that $\rho : G \to {\rm GL}(V)$ is an $n$-dimensional 
representation that is not $(i, j)$-thick. 
If $\min\{ i,  j\} \le m \le n - \max\{ i, j \}$ or $\max\{ i,  j\} \le m \le n - \min\{ i, j \}$, then 
$\rho$ is not $m$-thick. In particular, $\rho$ is not thick.  
\end{proposition}

\proof
Suppose that $\rho$ is not $(i, j)$-thick with $i+j \le n$. 
There exist subspaces $V_1$ and $V_2$ of $V$ 
with $\dim_{k} V_1=i$ and $\dim_{k} V_2=j$ such that 
$(\rho(g)V_1)\cap V_2\neq 0$ for any $g\in G$. 
We may assume that $i=\min \{ i, j\}$ and $j=\max\{ i, j\}$.  
For $i\le m \le n-j$, 
choose subspaces $V'_1 \supseteq V_1$ and $V'_2 \supseteq V_2$ with 
$\dim_{k} V'_1= m$ and $\dim_{k} V'_2=n-m$.  
Then $(\rho(g)V'_1)\cap V'_2\neq 0$ for any $g\in G$, which implies that $\rho$ is not 
$m$-thick. We can also prove the case $j\le m \le n-i$.  
\qed  

\begin{corollary}\label{cor:ijthick} 
Let $\rho : G \to {\rm GL}(V)$ be an $n$-dimensional representation of a group $G$ over a field $k$.  
Let $i$ and $j$ be positive integers with $\max\{i, j\} \le \dfrac{n}{2}$. Assume that 
$\rho$ is not $(i, j)$-thick. 
Then $\rho$ is not $m$-thick for any $\min\{i, j\} \le m \le n-\min\{i, j\}$. 
\end{corollary}

\proof
If $\max\{i, j\} \le \dfrac{n}{2}$, then 
the union of $\{ m \in {\Bbb Z} \mid \min\{ i,  j\} \le m \le n- \max\{ i, j \}\}$ and  
$\{ m \in {\Bbb Z} \mid \max\{ i,  j\} \le m \le n - \min\{ i, j \}\}$ coincides with  
$\{ m \in {\Bbb Z} \mid \min\{i, j\} \le m \le n-\min\{i, j\} \}$.  
The statement follows from Proposition~\ref{prop:ijthick}.  
\qed

\bigskip 

Although the notion of $(i, j)$-dense is not used in this paper, we include its definition here for reference.

\begin{definition}[{\cite{NO4}}]\rm\label{def:ijdense}
Let $0\le i, j \le n$.   
For an $n$-dimensional representation $\rho : G \to {\rm GL}(V)$,  
we say that $\rho$ is {\it $(i, j)$-dense} if $W_1\wedge W_2\neq 0$ for any 
$G$-invariant subspaces $0\neq W_1 \subseteq \bigwedge^i V$ and 
$0\neq W_2 \subseteq \bigwedge^j V$. Note that if $\rho$ is $(i, j)$-dense, then 
$0 \le i+j \le n$. 
For $0<m<n$, $\rho$ is $m$-dense if and only if it is $(m, n-m)$-dense. 
We also see that $\rho$ is $(i, j)$-dense if and only if it is $(j, i)$-dense. 
We can also prove that if $\rho : G \to {\rm GL}(V)$ is $(i+1, j)$-dense with 
$i+j < n$, then it is $(i, j)$-dense (for details, see \cite{NO4}). 
\end{definition} 

\begin{remark}[{\cite{NO4}}]\rm 
We see that $\rho : G \to {\rm GL}(V)$ is not $(i, j)$-thick if and only if 
there exist non-zero $G$-invariant realizable subspaces  $0\neq W_1 \subseteq \bigwedge^i V$ and $0\neq W_2 \subseteq \bigwedge^j V$ such that 
$W_1 \wedge W_2=0$. 
Hence, if $\rho$ is $(i, j)$-dense, then it is $(i, j)$-thick. 
\end{remark}

\begin{definition}[{\cite{NO4}}]\rm\label{def:thickdiagram}  
For a positive integer $n$, set $\Delta(n) = \{ (i, j)\in {\Bbb Z}^2 \mid i, j \ge 0 \text{ and } i+j \le n \}$.  
For an $n$-dimensional representation $\rho : G \to {\rm GL}(V)$, 
we define the {\it thick diagram} $T(\rho)$ of $\rho$ by 
\[
T(\rho) = \{ (i, j)\in \Delta(n) \mid \rho \text{ is $(i, j)$-thick} \}.  
\] 
We also define the {\it dense diagram} $D(\rho)$ of $\rho$ by 
\[
D(\rho) = \{ (i, j)\in \Delta(n) \mid \rho \text{ is $(i, j)$-dense} \}.  
\] 
It is easy to see that 
\[
\{ (i, 0) \mid 0 \le i  \le n \}\cup \{ (0, j) \mid 0 \le j \le n \} \subseteq D(\rho) \subseteq T(\rho) \subseteq \Delta(n).  
\]
We also see that $\rho$ is thick (resp. dense) if and only if $T(\rho)=\Delta(n)$ (resp. 
$D(\rho)=\Delta(n)$). 
\end{definition}

\begin{remark}[{\cite{NO4}}]\rm\label{rem:thickdiagram}  
For a subset $S \subseteq \Delta(n)$, set ${}^t S = \{ (j, i) \in \Delta(n) \mid (i, j)\in S \}$. 
For $(i, j) \in \Delta(n)$, we also set $\overline{(i, j)} = \{ (k, l) \in \Delta(n) \mid 0 \le k \le i, 0 \le l \le j \}$. For an $n$-dimensional representation $\rho : G \to {\rm GL}(V)$,  
${}^t T(\rho) = T(\rho)$ and ${}^t D(\rho) = D(\rho)$.  We also see that 
$(i, j) \in T(\rho)$ (resp. $(i, j) \in D(\rho)$), then $\overline{(i, j)} \subseteq T(\rho)$ (resp. $\overline{(i, j)} \subseteq D(\rho)$).  
In particular, if $(i', j') \in \overline{(i, j)}$ and $\rho$ is not $(i', j')$-thick, then 
$\rho$ is not $(i, j)$-thick. 
For further properties of $T(\rho)$ and $D(\rho)$, see \cite{NO4}. 
\end{remark}

\bigskip 

The classification of thick representations was first successfully completed in the case of connected semi-simple Lie groups over ${\Bbb C}$. 
Let $G$ be a connected semi-simple Lie group over ${\Bbb C}$. Let ${\frak g}, {\frak h}, \Delta^{+} (\subset {\frak h}^{\ast})$ be the Lie algebra of $G$, a Cartan subalgebra, the set of positive roots, respectively. 
For a finite-dimensional representation 
$\rho : G \to {\rm GL}(V)$ over ${\Bbb C}$, denote by $W(V)$ the set of weights of $V$. 
We can regard $W(V)$ as a partially ordered set by using $\Delta^{+}$. 
We say that $\rho : G \to {\rm GL}(V)$ is {\it weight multiplicity-free} if 
the dimension of the $\varphi$-eigenspace is $1$ 
for any $\varphi \in W(V)$. 

\begin{theorem}[{\cite[Theorem~3.5]{NO2}}]\label{th:WMF}
For a connected semi-simple Lie group $G$ over ${\Bbb C}$, a finite-dimensional irreducible representation 
$\rho : G \to {\rm GL}(V)$ over ${\Bbb C}$ is thick if and only if $\rho$ is weight multiplicity-free and the weight poset 
$W(V)$ is a totally ordered set. 
\end{theorem}

\begin{theorem}[{\cite[Theorem~3.6]{NO2}}]\label{th:listofthick}
The thick representations of connected complex 
simple Lie groups are those on the following list: 
\begin{enumerate}
\item the trivial $1$-dimensional representation for any groups, 
\item $A_n \; (n\ge 1)$    
\begin{itemize} 
\item the standard representation $V$ for $A_n \; (n\ge1)$ with highest weight $\omega_1$, 
\item the dual representation $V^{\ast}$ of $V$ for $A_n \; (n\ge1)$ with highest weight $\omega_n$, 
\item the symmetric tensor $S^{m}(V)$ $(m\ge 2)$ of $V$ for $A_1$ with highest weight $m\omega_1$,   
\end{itemize}
\item $B_n \; (n\ge 2)$ 
\begin{itemize}
\item  the standard representation $V$ for $B_n \; (n\ge2)$ with highest weight $\omega_1$,  
\item  the spin representation for $B_2$ with highest weight $\omega_2$,  
\end{itemize} 
\item $C_n \; (n\ge 3)$ 
\begin{itemize}
\item  the standard representation $V$ for $C_n \; (n\ge3)$ with highest weight $\omega_1$,  
\end{itemize} 
\item $G_2$ 
\begin{itemize}
\item the $7$-dimensional representation $V$ for $G_2$ with highest weight $\omega_1$.   
\end{itemize} 
\end{enumerate}
\end{theorem}  

\begin{theorem}[{\cite[Theorem~3.7]{NO2}}]\label{th:listofdense} 
The dense representations of connected complex simple Lie groups are those on the following list: 
\begin{enumerate}
\item the trivial $1$-dimensional representation for any groups,  
\item $A_n \; (n\ge 1)$    
\begin{itemize} 
\item the standard representation $V$ for $A_n \; (n\ge1)$ with highest weight $\omega_1$,  
\item the dual representation $V^{\ast}$ of $V$ for $A_n \; (n\ge1)$ with highest weight $\omega_n$,  
\item the symmetric tensor $S^{2}(V)$ of $V$ for $A_1$ with highest weight $2\omega_1$,   
\end{itemize}
\item $B_n \; (n\ge 2)$ 
\begin{itemize}
\item  the standard representation $V$ for $B_n \; (n\ge2)$ with highest weight $\omega_1$.  
\end{itemize} 
\end{enumerate}
\end{theorem} 

\begin{remark}\rm 
Through the theorems above, we see that irreducible, thick, and dense representations do not necessarily coincide. 
However, for the symmetric group, thick representations and dense representations coincide by 
Theorem~\ref{th:main1}.  
\end{remark}

By the following theorems, we cannot essentially obtain any new thick representations using $\bigwedge^m$, $S^m$, or $\otimes_{k}$. 

\begin{theorem}[{\cite{NO4} and \cite[Theorems~11 and 12]{NO3}}]\label{th:exteriorsymmetric}
Let $\rho : G \to {\rm GL}(V)$ be an $n$-dimensional representation of a group $G$ over a field $k$.  
\begin{enumerate}
\item\label{item:exterior} If $n\ge 4$ and $2\le m \le n-2$, then the exterior product 
$\bigwedge^m V$ is not $3$-thick as 
a representation of $G$.  
\item\label{item:symmetric} If $n\ge 3$ and $m\ge 2$, then the symmetric tensor 
$S^m(V)$ is not $3$-thick as a representation of $G$.  
\end{enumerate}
\end{theorem}

\proof 
Here, we prove only (\ref{item:exterior}). For the proof of (\ref{item:symmetric}), see \cite{NO4}. 
Let $\{ e_1, e_2, \ldots, e_n \}$ be a basis of $V$ over $k$. 
Then ${\mathcal B} = \{ e_{i_1} \wedge e_{i_2} \wedge \cdots \wedge e_{i_m} \mid 1 \le i_1 < i_2 < \cdots < i_m \le n \}$ 
is a basis of $\bigwedge^m V$. We define subsets $S_1$ and $S_2$ of ${\mathcal B}$ by 
\[
S_1  =   \{ e_1 \wedge e_2 \wedge \cdots \wedge e_{m-1} \wedge e_i \mid i=m, m+1, m+2 \}  
\] and 
\begin{multline*} 
 S_2  =  \{ e_1 \wedge e_2 \wedge \cdots \wedge e_{m-1} \wedge e_m,  \; e_1 \wedge e_2 \wedge \cdots \wedge e_{m-1} \wedge e_{m+1},  \; 
 e_1 \wedge e_2 \wedge \cdots \wedge e_{m-2} \wedge e_{m} \wedge e_{m+1} \},  
\end{multline*}
respectively. 
Let $W_1$ be the subspace of $\bigwedge^m V$ spanned by $S_1$ and $W_2$ the subspace 
of $\bigwedge^m V$ spanned by 
${\mathcal B} \setminus S_2$.  Note that $\dim_{k} W_1=3$ and $\dim_{k} W_2 = \binom{n}{m} -3$. 

Let us show that $(\rho(g)W_1)\cap W_2 \neq 0$ for any $g \in G$. Setting $x_i = e_1 \wedge e_2 \wedge \cdots \wedge e_{m-1} \wedge e_{m+i-1} \in W_1$ for $i=1, 2, 3$, we have $W_1 = kx_1 \oplus kx_2 \oplus kx_3$.  
As in Definition~\ref{def:vector}, we 
denote by $\bigwedge^3 (\rho(g)W_1)$ and $\bigwedge^{\binom{n}{m}-3} W_2$ the non-zero vectors of 
$\bigwedge^3 (\bigwedge^m V)$ and $\bigwedge^{\binom{n}{m}-3} (\bigwedge^m V)$ defined by $\rho(g)W_1$ and $W_2$  (up to scalar multiplication), respectively. We may write $\bigwedge^3 (\rho(g)W_1) = \rho(g)x_1\wedge \rho(g)x_2\wedge \rho(g)x_3$.  
For $1 \le i \le m+2$, we can write 
\[
\rho(g)e_i = f_i +h_i,  
\]
where $f_i \in \oplus_{j=1}^{m+1} ke_j$ and $h_i \in  \oplus_{j=m+2}^{n} ke_j$. 
Note that if some $i_j$ of $\{ i_1, i_2, \ldots, i_m \} \subset \{ 1, 2, \ldots ,n \}$ is larger than $m+1$, then $e_{i_1} \wedge e_{i_2} \wedge \cdots \wedge e_{i_m} \in W_2$. In particular, 
\[
\rho(g)x_i = f_1\wedge f_2 \wedge \cdots \wedge f_{m-1} \wedge f_{m+i-1} + w_i
\]
for $i=1, 2, 3$, where $w_i \in W_2$. 
Then 
\begin{eqnarray*}
& & \bigwedge^3(\rho(g)W_1) \wedge (\bigwedge^{\binom{n}{m}-3} W_2) \\
 & = & (f_1\wedge \cdots \wedge f_{m-1} \wedge f_{m} )\wedge (f_1\wedge \cdots \wedge f_{m-1} \wedge f_{m+1} )  \wedge (f_1\wedge \cdots \wedge f_{m-1} \wedge f_{m+2} ) \\ 
 &  & \hspace*{60ex} \wedge  (\bigwedge^{\binom{n}{m}-3} W_2). 
\end{eqnarray*}
On the other hand, since $f_1, f_2, \ldots, f_{m+2} \in \oplus_{j=1}^{m+1} ke_j$ and $\dim_k (\oplus_{j=1}^{m+1} ke_j) =m+1$,  $f_1, f_2, \ldots, f_{m+2}$ are linearly dependent in $V$. 
It is easy to see that 
\[
(f_1\wedge \cdots \wedge f_{m-1} \wedge f_{m} )\wedge (f_1\wedge \cdots \wedge f_{m-1} \wedge f_{m+1} ) \wedge (f_1\wedge \cdots \wedge f_{m-1} \wedge f_{m+2} )  = 0
\]
in $\bigwedge^3(\bigwedge^m V)$.  
This implies that $\bigwedge^3(\rho(g)W_1) \wedge (\bigwedge^{\binom{n}{m}-3} W_2) = 0$.  
Hence, $(\rho(g)W_1)\cap W_2 \neq 0$ for any $g \in G$. 
Therefore, $\bigwedge^m V$ is not $3$-thick. 
\qed

\begin{theorem}[{\cite{NO4} and \cite[Theorem~10]{NO3}}] 
Let $V, W$ be $G$-modules over a field $k$. If $\dim_{k} V\ge2$ and $\dim_{k} W\ge 2$, then 
$V\otimes_{k} W$ is not $2$-thick as a representation of $G$.  
\end{theorem}

\begin{lemma}[{\cite{NO4}}]\label{lemma:grouphomnon-thick}
Let $\varphi : G' \to G$ be a group homomorphism. 
Let $\rho : G \to {\rm GL}(V)$ be a finite-dimensional representation over a field $k$.  
If $V$ is not thick (resp. not dense), then neither is $\varphi^{\ast} V$. 
Moreover, if $\varphi$ is a group isomorphism, then $V$ is thick (resp. dense) if and only if  so is $\varphi^{\ast} V$. 
\end{lemma}

\proof 
The statement is immediate from the definitions of thickness and denseness. 
\qed 

\bigskip 

Lemmas~\ref{lemma:binomineq-1} and \ref{lemma:binomineq} will be used in Sections~\ref{section:s=2} and~\ref{section:s=3}, and in Section~\ref{section:denserepresentations}, respectively.

\begin{lemma}\label{lemma:binomineq-1} 
For $0 \le r_1 \le r_2 \le \left[\dfrac{n}{2}\right]$, the following inequality holds: 
\[
\binom{n}{r_1} \le \binom{n}{r_2}. 
\] 
\end{lemma} 

\proof
Since 
\begin{align*}
\binom{n}{r_1} \le \binom{n}{r_2} & \Longleftrightarrow \dfrac{n!}{r_1!(n-r_1)!} \le \dfrac{n!}{r_2!(n-r_2)!} \\
& \Longleftrightarrow \dfrac{r_2!}{r_1!} \le \dfrac{(n-r_1)!}{(n-r_2)!} \\
& \Longleftrightarrow \overbrace{r_2(r_2-1)\cdot\cdots\cdot (r_1+1)}^{(r_2-r_1)\mbox{-terms}} \le \overbrace{(n-r_1)(n-r_1-1)\cdot\cdots\cdot (n-r_2+1)}^{(r_2-r_1)\mbox{-terms}}
\end{align*} 
and $r_2-k \le n-r_1-k$ for $0\le k \le r_2-r_1-1$,   
the inequality holds. 
\qed 

\begin{lemma}\label{lemma:binomineq} 
The following inequalities hold: 
\begin{align}
\binom{2m}{m} & \ge 2^{m}  \quad (m\ge1), \label{ineq:even} \\
\binom{2m+1}{m} & \ge 2^{m+1}  \quad (m\ge 2). \label{ineq:odd}  
\end{align}
\end{lemma}

\proof
The inequality (\ref{ineq:even}) can be proved in the same way, so here we only prove (\ref{ineq:odd}). 
Let us consider how to choose $m$ elements from the set 
\[ 
 X=\{a_1, b_1, a_2, b_2, \ldots, a_m, b_m, c\}. 
\]   
Since there are $m$-element subsets 
$\{ \ast_i \mid  1\le i \le m \}$ ($\ast= a\mbox{ or } b$) and $X_j = \{c \} \cup \{ \ast_i \mid  1\le i \le m, i \neq j \}$ ($1\le j \le m$, $\ast= a\mbox{ or } b$) of $X$,  we obtain 
\[
\binom{2m+1}{m} \ge 2^m+m\cdot 2^{m-1} \ge 2^{m+1}
\] 
for $m \ge 2$. This completes the proof.  
\qed 

\section{Review of Specht modules}\label{section:Specht} 
In this section, we review Specht modules (for details, see \cite{Peel}).  Since the irreducible representations of the symmetric group can be treated concretely via 
Specht modules, we complete the classification of thick representations of the symmetric group in Sections~\ref{section:thickrepresentations}--\ref{section:sge4}.   

\begin{definition}[{{\it cf.} \cite{Iwahori} and \cite{Fulton}}]\rm
Let $n$ be a positive integer. 
To each partition $\lambda= (m_1, m_2, \ldots, m_s)$ of $n$ with $m_1\ge m_2\ge\cdots \ge m_s>0$ and  $n=m_1+m_2+\cdots +m_s$, we associate a Young diagram as follows.  
For example, for $\lambda=(4, 2, 1)$, 
the Young diagram is 
\[
{\small\begin{ytableau}
       \phantom{a}  & \phantom{a} & \phantom{a} & \phantom{a} \\
       \phantom{a}  & \phantom{a} \\
       \phantom{a}  \\
\end{ytableau}}\quad .  
\]
We denote $n=m_1+m_2+\cdots +m_s$ by $|\lambda|$.  
\end{definition}

\begin{definition}[{{\it cf.} \cite{Iwahori}, \cite{Fulton}, and \cite{JamesKerber}}]\label{def:lambda-tableau}\rm
For a partition  $\lambda= (m_1, m_2, \ldots, m_s)$ of $n$, 
a {\it Young tableau}, or a {\it $\lambda$-tableau}, is a tableau in which the entries are the numbers from $1$ to $n$, 
each occurring once. 
A {\it standard Young tableau}, or a {\it standard $\lambda$-tableau}, is a Young tableau which satisfies 
\begin{enumerate}
\item strictly increasing across each row,  
\item strictly increasing down each column. 
\end{enumerate}
For example, for the partition $\lambda=(4, 2, 1)$ of $7=4+2+1$, let us 
consider the following Young tableaux 
\[
(a)\; 
{\small\begin{ytableau}
       3 & 1  & 6 & 2\\
       7  & 5 \\
       4  \\
\end{ytableau}}, \quad 
(b)\;  
{\small\begin{ytableau}
       1  & 2  & 5 & 6\\
       3  & 4 \\
       7  \\
\end{ytableau}}. 
\]
(a) is a Young tableau, but not standard. 
(b) is a standard Young tableau.  
(Note that the usage of the term 
$\lambda$-tableau is slightly different from that in \cite{Peel}.) 
\end{definition}

\begin{definition}[{\cite[pages~88--89]{Peel}}]\rm
Suppose that $a_1, \ldots, a_t \in \{1, 2, \ldots, n\}$ occur in the $j$-th column of the $\lambda$-tableau $y$, with $a_i$ in the $i$-th row. We form the difference product 
\[
\Delta(x_{a_1}, \ldots, x_{a_t}) = \prod_{i<j} (x_{a_i}-x_{a_j})
\]
if $t>1$, and $\Delta(x_{a_1})=1$ if $t=1$. 
Multiplying these difference products for all the columns of $y$, we obtain a polynomial 
$f(y) \in {\Bbb Z}[x_1, \ldots, x_n]$, called a {\it Specht polynomial}. 
For example, if $y$ is (a) or (b) in Definition~\ref{def:lambda-tableau},  
then 
\begin{align*}
(a)\; &f(y)=(x_3-x_7)(x_3-x_4)(x_7-x_4)(x_1-x_5)  \quad \mbox{ and } \\
(b)\; &f(y)=(x_1-x_3)(x_1-x_7)(x_3-x_7)(x_2-x_4). 
\end{align*} 
\end{definition}

\begin{definition}[{\cite[page~89]{Peel}}]\rm
For a partition $\lambda$ of $n$, let 
\[ S^{\lambda}  = {\mathbb Z}\{ f(y) \mid y \mbox{ is a $\lambda$-tableau }\} 
\subset {\mathbb Z}[x_1, \ldots, x_n]. 
\]
Then $S_n$ canonically acts on $S^{\lambda}$. We call $S^{\lambda}$ the {\it Specht module} corresponding to 
$\lambda$. 
\end{definition}

\begin{theorem}[{\cite[Theorem~1.1]{Peel}}]\label{th:basisofSpecht} 
For a partition $\lambda$ of $n$, $S^{\lambda}$ has a ${\Bbb Z}$-basis 
\[
B^{\lambda} = \{ f(y) \mid y \mbox{ is a standard $\lambda$-tableau }\}.  
\]
\end{theorem}

\begin{notation}\rm 
Let ${\Bbb K}$ be a field of characteristic zero. 
For a partition $\lambda$ of $n$, we denote by 
$V_{\lambda}$ the Specht module $S^{\lambda}\otimes_{{\Bbb Z}}{\Bbb K}$. 
\end{notation}

\begin{remark}\rm In \cite[Section~7.2]{Fulton}, 
for a partition $\lambda$ of $n$, the irreducible representation $S^{\lambda}$ of $S_n$ is 
constructed using tabloids, and this is also called a Specht module. It is isomorphic to the Specht module in our sense, following \cite{Peel}.
\end{remark}

\begin{theorem}[{\cite[page~90 and Section~4]{Peel} and \cite[Section~7.2, Proposition~1]{Fulton}}]
For each partition $\lambda$ of $n$, $V_{\lambda}$ 
is an irreducible representation of $S_n$ over a field ${\Bbb K}$ of characteristic zero. 
Every irreducible representation of $S_n$ over ${\Bbb K}$ is 
isomorphic to exactly one $V_{\lambda}$. 
\end{theorem}

\begin{definition}\rm
For a partition $\lambda$ of $n$, we set 
\begin{align*} 
Y_{\lambda} & = \{ \lambda\mbox{-tableau } \},   \\ 
Y^{\rm st}_{\lambda} & = \{ \mbox{ standard } \lambda\mbox{-tableau } \}.  
\end{align*} 
We denote by $d(\lambda)$ the dimension of 
$V_{\lambda}$. 
Note that $d(\lambda) = \sharp Y^{\rm st}_{\lambda}$ by Theorem~\ref{th:basisofSpecht}. 
\end{definition}

\begin{proposition}\label{prop:formuladim}
For $\lambda = (m_1, m_2)$, 
\begin{align}
d(m_1, m_2) & = \dfrac{(m_1+m_2)!(m_1-m_2+1)}{(m_1+1)!m_2!}. \label{eq:dm1m2} 
\end{align}
For $\lambda = (m_1, m_2, m_3)$, 
\begin{align}
d(m_1, m_2, m_3) & = \dfrac{(m_1+m_2+m_3)!(m_1-m_2+1)(m_1-m_3+2)(m_2-m_3+1)}{(m_1+2)!(m_2+1)!m_3!}. \label{eq:dm1m2m3} 
\end{align}
\end{proposition}

\proof
By \cite[page~50, Hook Length Formula~4.12]{FH}, we can calculate 
$d(m_1, m_2)$ and $d(m_1, m_2, m_3)$.   
\qed 

\bigskip 


Let us consider a $\lambda$-tableau 
\[
y= {\small\begin{ytableau}
       \cdots & a_1 & \cdots  & b_1 & \cdots \\
       \cdots & a_2  & \cdots & b_2 & \cdots  \\
        & \vdots & & \vdots & \vdots \\
        \cdots & a_s & \cdots & b_s \\  
        & \vdots \\
        \cdots & a_t   
\end{ytableau}} 
\]
with $t\ge s$ and $\{a_1, \ldots, a_t, b_1, \ldots, b_s\} \subseteq \{1, \ldots, n\}$ such that 
the columns $k$ and $l$ with $k<l$ contain $\{a_1, \ldots, a_t\}$ and 
$\{b_1, \ldots, b_s \}$, respectively.  
For $1\le q \le s$, denote by $S(q)$ the subgroup of $S_n$ consisting of permutations of 
$\{ a_q, \ldots, a_t, b_1, \ldots, b_q\}$. 
Let $C(y)$ be the set of all permutations $\sigma \in S_n$ such that for each $i$, $\sigma(i)$ and $i$ occur in the same column of $y$. Such a permutation is called a {\it column permutation} of $y$.  We also denote by $R(y)$ the subgroup of $S_n$ consisting of all {\it row permutations} of $y$.  
For a subset $T$ of $S_n$, we define $\alpha(T) = \sum_{\sigma\in T} ({\rm sgn} \sigma) 
 \sigma \in {\Bbb Z}[S_n]$. 

By \cite{Peel}, we have the following results: 

\begin{theorem}[{\cite[Theorem~3.1]{Peel}}]\label{th:PeelalphaSfy=0}
Let $y$ and $S(q)$ be as above. Then 
$\alpha(S(q)) f(y)=0$ in ${\Bbb Z}[x_1, \ldots, x_n]$. 
\end{theorem}

\begin{lemma}[{\cite[Lemma~3.2]{Peel}}]\label{lemma:UValpha}
If $T=UV$ where $U$ and $V$ are subgroups of $S_n$, then 
$\alpha(U)\alpha(V)=\sharp (U\cap V) \alpha(T)$ in ${\Bbb Z}[S_n]$. 
\end{lemma} 
 
\bigskip 

We define a subgroup $H(q)$ of $S(q)$ by 
\[
H(q) = \{ \sigma \in S(q) \mid \sigma f(y) = ({\rm sgn}\sigma) f(y) \}.  
\] 
Then $H(q)=S(q)\cap C(y)$. 
Let $W(q)$ be a set of left coset representatives of $H(q)$ in $S(q)$ such that 
$W(q)$ contains the identity permutation $1$. 

\begin{theorem}[{\cite[Theorem~3.4]{Peel}}]\label{th:Garnir's relation}
Let $y$ and $W(q)$ be as above. Then 
\[
f(y)=-\sum_{\nu \in W(q)\setminus\{1\}} ({\rm sgn} \nu) f(\nu y)
\]  
in ${\Bbb Z}[x_1, \ldots, x_n]$. 
\end{theorem} 

We call the relation in Theorem~\ref{th:Garnir's relation} 
the {\it Garnir relation}.  This relation is one of the relations established by 
Garnir in \cite[Section~11]{Garnir}, and it is used in the proof of Theorem~\ref{th:basisofSpecht} (\cite[Theorem~1.1]{Peel}).  

\begin{example}\rm 
By the Garnir relation, we have 
\begin{align*} y = 
\begin{ytableau}
1 & 3 & 5 \\
2 & 4 
\end{ytableau}
& =
\begin{ytableau}
1 & 5 & 3 \\
2 & 4 
\end{ytableau}
+ 	
\begin{ytableau}
1 & 3 & 4 \\
2 & 5 
\end{ytableau}  
\end{align*}
in ${\mathbb Z}[x_1, \ldots, x_5]$, where 
we 
omit $f$. More precisely, the columns $k=2$ and $l=3$ of $y$ coincide with $\{3, 4 \}$ and $\{ 5 \}$, respectively. For $q=1$, $S(1)$ is the subgroup of $S_5$ consisting of permutations $\{3, 4, 5\}$.  
Then $H(1)= \{ 1, (3, 4) \}$ and we can take $W(1) = \{ 1, (3, 5), (4, 5) \}$. 
By Theorem~\ref{th:Garnir's relation}, we obtain the equality above. 
\end{example}

\begin{remark}\rm
For a $\lambda$-tableau $y$, we have $\sigma f(y)= ({\rm sgn} \sigma) f(y)$ for any 
$\sigma \in C(y)$. 
If $y_2$ is obtained by changing the columns $k$ and $l$ of $y_1$, in other words,  
\begin{align*} 
y_1= {\small\begin{ytableau}
       \cdots & a_1 & \cdots  & b_1 & \cdots \\
       \cdots & a_2  & \cdots & b_2 & \cdots  \\
        \vdots & \vdots & & \vdots & \vdots \\
        \cdots & a_s & \cdots & b_s \\  
        \vdots \\
        \cdots    
\end{ytableau}} \mbox { and } 
y_2= {\small\begin{ytableau}
       \cdots & b_1 & \cdots  & a_1 & \cdots \\
       \cdots & b_2  & \cdots & a_2 & \cdots  \\
        \vdots & \vdots & & \vdots & \vdots \\
        \cdots & b_s & \cdots & a_s \\  
        \vdots \\
        \cdots    
\end{ytableau}} , 
\end{align*} 
then $f(y_1) = f(y_2)$ in ${\Bbb Z}[x_1, \ldots, x_n]$. 
By these properties, we can obtain 
\begin{align*}  
\begin{ytableau}
3 & 2 & 5 \\
4 & 1
\end{ytableau}
=
\begin{ytableau}
2 & 3 & 5 \\
1 & 4 
\end{ytableau}
= - \: 
\begin{ytableau}
1 & 3 & 5 \\
2 & 4 
\end{ytableau}  
\end{align*}
in ${\mathbb Z}[x_1, \ldots, x_5]$, where we 
omit $f$.  
\end{remark} 

\begin{remark}\rm
Let $\lambda$ be a partition of a positive integer $n$. 
For a $\lambda$-tableau $y$, set 
\[
a_{y} = \sum_{\sigma \in R(y)} \sigma, \quad b_{y} = \sum_{\sigma \in C(y)} ({\rm sgn} \sigma)  \sigma \;\in {\Bbb Z}[S_n].  
\] 
Then ${\Bbb K}[S_n] a_y b_y \cong {\Bbb K}[S_n] b_y a_y$ is an irreducible representation of 
$S_n$ over ${\Bbb K}$ ({\it cf}. \cite{Iwahori} and \cite[{pages~46--47}]{FH}). 
Note that it is isomorphic to the Specht module $V_{\lambda}$ over ${\Bbb K}$.  
\end{remark}

\section{Dense representations of the symmetric group}\label{section:denserepresentations} 
In this section, we classify the dense representations of the symmetric group. 
To classify dense representations, we prepare the following lemma.  

\begin{lemma}\label{lemma:sqrtn!isnotdense} 
Let $V_{\lambda}$ be an irreducible representation of $S_n$ over ${\Bbb K}$. 
If there exists $1 \le r  \le \dim_{{\Bbb K}} V_{\lambda}- 1$ such that 
$\dim_{{\Bbb K}} \bigwedge^r V_{\lambda} > \sqrt{n!}$, then $V_{\lambda}$ is not dense. 
\end{lemma}

\proof 
Note that $\dim_{{\Bbb K}} V_{\mu} \le \sqrt{n!}$ for any irreducible representation $V_{\mu}$ of $S_n$. Indeed, 
\[
 (\dim_{{\Bbb K}}V_{\mu})^2 \le 
\sum_{\rho} (\dim_{{\Bbb K}} V_{\rho})^2 
= \sharp S_n = n!  
\]
by the Wedderburn-Artin theorem.   Hence,  $\dim_{{\Bbb K}} V_{\mu} \le \sqrt{n!}$. 
If $\dim_{{\Bbb K}} \bigwedge^r V_{\lambda} > \sqrt{n!}$, then $\bigwedge^r V_{\lambda}$ is not irreducible. 
Therefore, $V_{\lambda}$ is not dense. 
\qed

\begin{theorem}\label{th:main3}
The dense representations of the symmetric group $S_n$ over a field ${\Bbb K}$ of characteristic zero are those on the following list: 
\begin{enumerate}
\item\label{item:Sndenserep1} the trivial representation $V_{(n)}$ of $S_n$ for $n\ge 1$, 
\item\label{item:Sndenserep2} the sign representation $V_{(1^n)}$ of $S_n$ for $n\ge 2$,  
\item\label{item:Sndenserep3} the standard representation $V_{(n-1, 1)}$ of $S_n$ for $n \ge 3$,  
\item\label{item:Sndenserep4} the product of the standard and sign representation $V_{(2, 1^{n-2})}$ of $S_n$ for $n \ge 4$,  
\item\label{item:Sndenserep5} the $2$-dimensional irreducible representation $V_{(2^2)}$ of $S_4$, 
\item\label{item:Sndenserep6} the $5$-dimensional irreducible representations $V_{(2^3)}$ and $V_{(3^2)}$ of $S_6$.     
\end{enumerate}
\end{theorem}

\proof
Although the proof given here uses the results over ${\Bbb C}$, it can be 
proved in the same way over any field ${\Bbb K}$ of characteristic zero. 
First, we show that any representations in 
(\ref{item:Sndenserep1})--(\ref{item:Sndenserep6}) are dense.
Since 
$V_{(n)}$ and 
$V_{(1^n)}$ are $1$-dimensional,  they are dense. 
By \cite[Exercise~4.6]{FH},  $\bigwedge^r V_{(n-1, 1)} = V_{(n-r, 1^r)}$,   
and hence 
$V_{(n-1, 1)}$ is dense. 
Using $V_{(2, 1^{n-2})} = V_{(n-1, 1)}\otimes_{{\Bbb K}} V_{(1^n)}$,  we see that 
$V_{(2, 1^{n-2})}$ is also dense. 
The $2$-dimensional irreducible representation $V_{(2^2)}$ is dense by Corollary~\ref{cor:whendimle3}. 
For (\ref{item:Sndenserep6}), let us consider an outer automorphism $\varphi : S_6 \to S_6$ 
which is given in \cite{Lorimer}. 
Since $V_{(2^3)} = \varphi^{\ast} V_{(5,1)}$ and $V_{(3^2)} = \varphi^{\ast} V_{(2,1^4)}$, 
they are dense by Lemma~\ref{lemma:grouphomnon-thick}. 

Next, let us consider the case $n\ge 9$.  
By \cite[Result~2]{Rasala}, if $n \ge 9$, then the first four minimal 
dimensions of irreducible representations of 
$S_n$ are 
\[
\mbox{(A)} \;  1,\;\;\; \mbox{(B)} \;  n-1, \;\;\;  \mbox{(C)} \;  \frac{1}{2}n(n-3), \;\;\; \mbox{(D)} \;  \frac{1}{2}(n-1)(n-2). 
\]
The case (A) corresponds to (\ref{item:Sndenserep1}) and (\ref{item:Sndenserep2}), and the case (B) corresponds to (\ref{item:Sndenserep3}) and (\ref{item:Sndenserep4})  ({\it cf.} Remark~\ref{rem:Rasala}).  By the discussion above,  these representations  
(\ref{item:Sndenserep1})--(\ref{item:Sndenserep4}) are dense. 
We claim that any $d$-dimensional irreducible representation $V$ of $S_n\; (n \ge 9)$ is not dense 
if $d \ge \frac{1}{2}n(n-3) \ge \frac{1}{2}\cdot 9 \cdot (9-3) = 27$. 
We prove this claim separately for the cases $d$ even and $d$ odd. 

When $d = 2m$ for $m \in {\Bbb N}$, it follows from (\ref{ineq:even}) of  Lemma~\ref{lemma:binomineq} that 
\begin{eqnarray*}
\dim_{{\Bbb K}} \bigwedge^m V  & = & \binom{2m}{m} \\ 
 & \ge & 2^m \\
 & \ge & \sqrt{2^{\frac{1}{2}(n(n-3))}}.   
\end{eqnarray*}
It can be shown that $2^{\frac{1}{2}(n-3)} > n$ for $n \ge 10$ by induction. 
If $n \ge 10$, then 
\[
\sqrt{2^{\frac{1}{2}n(n-3)}} >  \sqrt{n^n}  >  \sqrt{n!}.    
\]
If $n=9$, then we can verify 
\[
\sqrt{2^{\frac{1}{2}n(n-3)}}  >  \sqrt{n!}.    
\]
by direct calculation. Hence, 
\[
\dim_{{\Bbb K}} \bigwedge^m V  > \sqrt{n!} 
\]
for $n\ge 9$, which implies that $V$ is not dense by Lemma~\ref{lemma:sqrtn!isnotdense}.  

When $d = 2m+1$ for $m \in {\Bbb N}$, it follows from (\ref{ineq:odd}) of  Lemma~\ref{lemma:binomineq} that 
\begin{eqnarray*}
\dim_{{\Bbb K}} \bigwedge^m V  & = & \binom{2m+1}{m} \\ 
 & \ge & 2^{m+1} \\
 & \ge & \sqrt{2^{\frac{1}{2}(n-1)(n-2)}},    
\end{eqnarray*}
since $m+1 = \frac{d+1}{2} \ge (\frac{n(n-3)}{2}+1)/2 =\frac{(n-1)(n-2)}{4}$ and $m=\frac{d-1}{2}\ge \frac{27-1}{2} =13$.  
It can be shown that $2^{\frac{1}{2}(n-2)} > n$ for $n \ge 9$ by induction. 
If $n \ge 9$, then 
\[
\dim_{{\Bbb K}} \bigwedge^m V \ge \sqrt{2^{\frac{1}{2}(n-1)(n-2)}} >  \sqrt{n^{n-1}}  >  \sqrt{n!}.      
\]
By Lemma~\ref{lemma:sqrtn!isnotdense},  
$V$ is not dense.    
Hence, any irreducible representation $V$ of $S_n \;  (n \ge 9)$ except the case 
(\ref{item:Sndenserep1})--(\ref{item:Sndenserep4}) is not dense. 

Finally, let us consider the case $n \le 8$. For details on irreducible representations of $S_n$ ($n\le 8$), see \cite[Appendix~I, I.A Character Tables]{JamesKerber}. 
If $n \le 4$, then any irreducible representations are dense, which are listed in 
(\ref{item:Sndenserep1})--(\ref{item:Sndenserep6}). 
Let ${\mathcal I}_n$ be the set of all (equivalence classes of) irreducible representations of $S_n$ that do not 
appear in  (\ref{item:Sndenserep1})--(\ref{item:Sndenserep6}). 
Set $\max(n) := \max_{V \in {\mathcal I}_n} \dim_{{\Bbb K}} V$ and $\min(n) := 
\min_{V \in {\mathcal I}_n} \dim_{{\Bbb K}} V$. 
Note that any irreducible representation of $S_n$ has dimension $\le \max(n)$ for $5  \le n\le 8$. 
When $n=5$, ${\mathcal I}_5 = \{ V_{(3, 2)}, V_{(3, 1^2)}, V_{(2^2, 1)} \}$, $\max(5)=6$, and $\min(5)=5$. 
Since $\binom{6}{2} > \binom{5}{2} = 10 >6=\max(5)$,  any irreducible representation in ${\mathcal I}_5$ is not dense.  
When $n=6$, 
\[
{\mathcal I}_6 = \{ V_{(2^2, 1^2)}, V_{(3, 1^3)}, V_{(3, 2, 1)}, V_{(4, 1^2)}, V_{(4, 2)} \}, 
\]   
$\max(6)=16$, and $\min(6)=9$. 
Since $\binom{9}{2} = 36 > 16 = \max(6)$,  any irreducible representation in ${\mathcal I}_6$ is not dense.  
Similarly, $\max(7)=35$, $\min(7)=14$, and $\binom{14}{2} = 91 > 35 =\max(7)$, which implies that  
any irreducible representation in ${\mathcal I}_7$ is not dense. 
We also have $\max(8) =90$, $\min(8) =14$, and 
$\binom{14}{2} = 91 > 90 = \max(8)$, which implies that  
any irreducible representation in ${\mathcal I}_8$ is not dense. 
Thus, we have verified the case $n \le 8$. 

Summarizing the discussion above, we have shown the statement.  
\qed

\begin{remark}\label{rem:Rasala}\rm
Let $n\ge 7$, and 
let $\lambda$ be a partition of $n$ different from $(n), (1^n), (n-1, 1)$, and $(2, 1^{n-2})$. 
As stated in \cite[{page~145 $(2)^{\ast}$}]{Rasala}, we have $\dim_{\Bbb K} V_{\lambda} \ge n$.  
If $n\ge 9$, then $\dim_{\Bbb K} V_{\lambda} \ge \dfrac{1}{2} n(n-3)$ by \cite[Result~2]{Rasala}. 
\end{remark}

\section{Thick representations of the symmetric group}\label{section:thickrepresentations}\label{section:thickrepresentations} 
In the following sections, we classify the thick representations of the symmetric group. 
Theorems~\ref{th:main1} and \ref{th:main2} follow from Theorem~\ref{th:main3} and 
the following theorem.  

\begin{theorem}\label{th:main4} 
If $V_{\lambda}$ is not contained in the list of Theorem~\ref{th:main3}, then 
$V_{\lambda}$ is not thick. 
\end{theorem}

For proving Theorem~\ref{th:main4}, we only need to show that $V_{\lambda}$ is not 
$(i, j)$-thick with $i+j \le \dim_{{\Bbb K}} V_{\lambda}$ if $V_{\lambda}$ is not contained in the list of Theorem~\ref{th:main3} by Proposition~\ref{prop:ijthick}.  
From now on, let us prove Proposition~\ref{prop:main5}, which is equivalent to 
Theorem~\ref{th:main4}. 

\begin{proposition}\label{prop:main5} 
If $V_{\lambda}$ is not contained in the list of Theorem~\ref{th:main3}, then 
there exist subspaces $W_1$ and $W_2$ of $V_{\lambda}$ such that 
\begin{enumerate}
\item $\sigma(W_1)\cap W_2 \neq 0$ for any $\sigma \in S_n$, 
\item $\dim_{{\Bbb K}} W_1+\dim_{{\Bbb K}} W_2 \le \dim_{{\Bbb K}} V_{\lambda}$. 
\end{enumerate} 
In particular, $V_{\lambda}$ is not $(\dim_{{\Bbb K}} W_1, \dim_{{\Bbb K}} W_2)$-thick.  
\end{proposition}


Suppose that $V_{\lambda}$ is not contained in the list of Theorem~\ref{th:main3} 
for a partition $\lambda=(m_1, m_2, \ldots, m_s)$ of $n$. 
Then $s\ge 2$.   
To prove Proposition~\ref{prop:main5},  
we divide the argument into three cases: (Case 1) $s=2$, (Case 2) $s=3$, and (Case 3) $s\ge 4$. 
More precisely, Case~1 corresponds to partitions $\lambda=(m_1, m_2)$ with 
$m_1\ge m_2\ge2$, excluding $\lambda=(2, 2)$ and $(3, 3)$. 
Case~2 corresponds to partitions $\lambda=(m_1, m_2, m_3)$ with 
$m_1\ge m_2\ge m_3 \ge 2$, excluding $\lambda=(2, 2, 2)$, or 
$\lambda=(m_1, m_2, 1)$ with $m_1 \ge m_2 \ge 1$, excluding 
$\lambda=(1, 1, 1), (2, 1, 1)$.  
Case~3 corresponds to partitions $\lambda=(m_1, m_2, \ldots, m_s)$ with $s\ge 4$ and  
$m_1\ge m_2\ge \cdots \ge m_s \ge 1$, excluding 
$\lambda=(1^n)$ and $(2, 1^{n-2})$.  
They 
will be proved in 
Sections~\ref{section:s=2}, \ref{section:s=3}, and \ref{section:sge4}, respectively. 

\begin{remark}\rm
In fact, we can prove Theorems~\ref{th:main1} and \ref{th:main2} without using \cite[Result~2]{Rasala}.  Theorems~\ref{th:main1} and \ref{th:main2} follow from 
Theorem~\ref{th:main4} and the fact that if $V_{\lambda}$ is contained in the list of Theorem~\ref{th:main3}, then it is dense.  
\end{remark} 

\bigskip 

Let us recall the notation from Section~\ref{section:Specht}.  
In the sequel, for a $\lambda$-tableau $y$, 
we simply denote  the polynomial $f(y) \in {\Bbb Z}[x_1, \ldots, x_n]$ by $y$ 
whenever no confusion arises.   
For a subset $T$ of $Y_{\lambda}$, we define $\langle T \rangle$ 
to be the subspace of $V_{\lambda}$ spanned by $\{ y \in T\}$.  

The rest of this section is devoted to presenting 
several non-thick representations of the symmetric group, which will be used in 
Sections~\ref{section:s=2} and \ref{section:s=3}.  

\begin{proposition}\label{prop:okuyama-omoda}
For $n\ge 5$, $V_{(n-2, 2)}$ is not $2$-thick. In particular, it is not thick.  
\end{proposition}

\proof
For integers $i, j$ with $2 \le i < j \le n$ and $(i, j) \neq (2, 3)$, we set 
\[ 
f_{ij} = \begin{ytableau}
	1 & a_{ij} & \cdots \\
	i & j
\end{ytableau} \in Y^{\rm st}_{(n-2, 2)},    
\]
where $a_{ij}$ is uniquely determined so that $f_{ij}$ is standard. 
Let \(W_1\) be the vector subspace of \(V_{(n-2,2)}\) spanned by all the \(f(y)\) with
\(y\) a standard tableau other than $f_{25}$ and $f_{34}$. 
%
We show that for any \(\sigma \in S_n,\)
\(W_1\) contains one of \(\sigma f_{24}, \sigma f_{34}\) or \(\sigma f_{34} - \sigma f_{24}.\) 
If \(W_2\) denotes the vector subspace of \(V_{(n-2,2)}\) spanned by \(f_{24}\) and \(f_{34},\) this means that
\(W_1 \cap \sigma W_2 \neq 0\) for all \(\sigma \in S_n,\) thus \(V_{(n-2,2)}\) is not $2$-thick.
Since \(f_{34} - f_{24} = \begin{ytableau}
		1 & 2 & \cdots\\
		4 & 3
	\end{ytableau}_,\) we readily see that it suffices to show the following claim.
\medskip  

\underline{Claim}. 
	Let \(a, b, c, d\) be distinct integers with \(1\le a,b,c,d \le n.\) Then for some permutation \((x,y,z,w)\) of 
	\((a,b,c,d),\) \(\begin{ytableau}
		x & z & \cdots\\
		y & w
	\end{ytableau} \in W_1.\)

To prove the claim, we may assume without loss of generality that \(a<b<c<d.\)

\bigskip 

\noindent 
\paragraph{\underline{Case: \((a,b) = (1,2)\)}}\leavevmode\par  

\smallskip 

\paragraph{Subcase: $c=3$}\leavevmode\par 

\smallskip 

	Since 
	\(\begin{ytableau}
			1 & 3 & \cdots\\
			2 & d
		\end{ytableau} = f_{2d} \in W_1\) if \(d \neq 5,\) only the case
	\((a,b,c,d) = (1,2,3,5)\) matters. But 
	\(\begin{ytableau}
		1 & 2 & \cdots\\
		3 & 5
	\end{ytableau} 
	= f_{35} \in W_1.\)

\medskip 

\noindent 
\paragraph{Subcase: \(c = 4\)}\leavevmode\par  

\smallskip 

	If \(d>5,\) 
	\(
		\begin{ytableau}
			1 & 4 & \dots\\
			2 & d
		\end{ytableau} 
		= f_{2d} - f_{24} \in W_1,
	\)
	while if \(d = 5,\)
	we alternatively have
	\(\begin{ytableau}
			1 & 2 & \dots\\
			4 & 5
	\end{ytableau}
	=f_{45} \in W_1\) 

\medskip 

\noindent 
	\paragraph{Subcase: \(c \ge 5\)}\leavevmode\par  
\smallskip 

	We have \(\begin{ytableau}
			1 & 2 & \dots\\
			c & d
	\end{ytableau} 
	=f_{cd} \in W_1.\) 

\smallskip 

\noindent 
\paragraph{\underline{Case: \((a,b) = (1,3)\)}}\leavevmode\par  
\smallskip

	Since 
	\(
		\begin{ytableau}
			1 & c & \cdots\\
			3 & d
		\end{ytableau}
		= f_{3d} - f_{3c},
	\)
	we have 
	\(
		\begin{ytableau}
			1 & c & \cdots\\
			3 & d
		\end{ytableau}\in W_1,\) if \(c\neq 4.
	\)
	If \(c = 4,\)
	we alternatively have
	\begin{align*}
		\begin{ytableau}
			1 & 4 & \cdots\\
			d & 3
		\end{ytableau}
		&=
		-\begin{ytableau}
			1 & 3 & \cdots\\
			4 & d
		\end{ytableau}
		+ 	
		\begin{ytableau}
			1 & 4 & \cdots\\
			3 & d
		\end{ytableau}\\
		&=
		-\begin{ytableau}
			1 & 2 & \cdots\\
			4 & d
		\end{ytableau}
		+\begin{ytableau}
			1 & 2 & \cdots\\
			4 & 3
		\end{ytableau}
		+ 	
		\begin{ytableau}
			1 & 2 & \cdots\\
			3 & d
		\end{ytableau}
		-
		\begin{ytableau}
			1 & 2 & \cdots\\
			3 & 4
		\end{ytableau}\\
		&=
		 -f_{4d} - f_{24} + f_{3d} \in W_1.
	\end{align*}

\noindent 
\paragraph{\underline{Case: \(a = 1, b\ge 4\)}}\leavevmode\par 
\smallskip 

	We have
	\(
		\begin{ytableau}
			1 & c & \cdots\\
			b & d
		\end{ytableau}
		= f_{bd} - f_{bc} \in W_1.
	\)

\bigskip 

\noindent 
\paragraph{\underline{Case: \(a = 2\)}}\leavevmode\par 
\smallskip 

	In this case, 
	instead of calculating
		\(\begin{ytableau}
			2 & c & \cdots\\
			b & d
		\end{ytableau},\)
	we have
		\begin{align*}
		\begin{ytableau}
			2 & c & \cdots\\
			d & b
		\end{ytableau}
		&=	
		\begin{ytableau}
			2 & 1 & \cdots\\
			d & b
		\end{ytableau}
		-
		\begin{ytableau}
			2 & 1 & \cdots\\
			d & c
		\end{ytableau}
		=
		\begin{ytableau}
			1 & 2 & \cdots\\
			b & d
		\end{ytableau}
		-
		\begin{ytableau}
			1 & 2 & \cdots\\
			c & d
		\end{ytableau} = f_{bd} - f_{cd} \in W_1.
		\end{align*}

\bigskip 

\noindent 
\paragraph{\underline{Case: \(a \ge 3\)}}\leavevmode\par 
\smallskip 

We have 
	\begin{align*}
				\begin{ytableau}
					a & c & \cdots\\
					b & d
				\end{ytableau}
				&= 
				\begin{ytableau}
					a & 1 & \cdots\\
					b & d
				\end{ytableau}
				-
				\begin{ytableau}
					a & 1 & \cdots\\
					b & c
				\end{ytableau}
				=
				\begin{ytableau}
					1 & a & \cdots\\
					d & b
				\end{ytableau}
				-
				\begin{ytableau}
					1 & a & \cdots\\
					c & b
				\end{ytableau}\\
				&=
				-\begin{ytableau}
					1 & b & \cdots\\
					a & d
				\end{ytableau}
				+
				\begin{ytableau}
					1 & a & \cdots\\
					b & d
				\end{ytableau}
				+
				\begin{ytableau}
					1 & b & \cdots\\
					a & c
				\end{ytableau}
				-
				\begin{ytableau}
					1 & a & \cdots\\
					b & c
				\end{ytableau}\\
				&=
				-\begin{ytableau}
					1 & 2 & \cdots\\
					a & d
				\end{ytableau}
				+\begin{ytableau}
					1 & 2 & \cdots\\
					a & b
				\end{ytableau}
				+
				\begin{ytableau}
					1 & 2 & \cdots\\
					b & d
				\end{ytableau}
				-
				\begin{ytableau}
					1 & 2 & \cdots\\
					b & a
				\end{ytableau}\\
				&+
				\begin{ytableau}
					1 & 2 & \cdots\\
					a & c
				\end{ytableau}
				-
				\begin{ytableau}
					1 & 2 & \cdots\\
					a & b
				\end{ytableau}
				-
				\begin{ytableau}
					1 & 2 & \cdots\\
					b & c
				\end{ytableau}
				+
				\begin{ytableau}
					1 & 2 & \cdots\\
					b & a
				\end{ytableau}\\
				&= -f_{ad} + f_{bd} + f_{ac} - f_{bc} \in W_1.
		\end{align*}
Thus, we have proved the claim, which implies that $V_{(n-2, 2)}$ is not $2$-thick. 
\qed


\begin{proposition}\label{prop:311and411}
For $n\ge 5$, $V_{(n-2, 1, 1)}$ is not $3$-thick. In particular, it is not thick. 
\end{proposition}

\proof 
By \cite[Exercise~4.6]{FH}, $\bigwedge^2 V_{(n-1, 1)} \cong V_{(n-2, 1, 1)}$. 
Since $\dim_{{\Bbb K}} V_{(n-1, 1)} =n-1 \ge 4$ for $n\ge 5$,  
$V_{(n-2, 1, 1)}$ is not $3$-thick by Theorem~\ref{th:exteriorsymmetric}.   
\qed 

\begin{prop}\label{prop:dualrep}
For a partition  $\lambda= (\lambda_1, \lambda_2, \ldots, \lambda_s)$ of $n$, 
let ${}^t \lambda$ denote the conjugate partition to $\lambda$.   
As representations of $S_n$, 
$V_{\lambda}$ is $m$-thick (or thick) if and only if so is $V_{{}^t\!\lambda}$.   
\end{prop}  

\proof 
By \cite[Exercise~4.4 (c)]{FH}, $V_{{}^t\!\lambda} \cong V_{\lambda}\otimes_{\Bbb K} V_{(1^n)}$, where 
$V_{(1^n)}$ is the sign representation.      
The statement follows from the definition of thickness. 
\qed 

\begin{proposition}\label{prop:221} 
For $n\ge 5$, $V_{(2, 2, 1^{n-4})}$ is not $2$-thick. In particular, it is not thick. 
\end{proposition}

\proof
Note that ${}^t \lambda=(2, 2, 1^{n-4})$ for $\lambda=(n-2, 2)$. 
The statement follows from Propositions~\ref{prop:okuyama-omoda} and 
\ref{prop:dualrep}. 
\qed

\section{The case $\lambda = (m_1, m_2)$ with $m_1\ge m_2 \ge 2$}
\label{section:s=2}  
In this section, we prove Proposition~\ref{prop:main5} in Case~1; namely, 
that $V_{(m_1, m_2)}$ is not thick if 
it is not contained in the list of Theorem~\ref{th:main3}. 
More precisely, if $\lambda=(m_1, m_2)$ with $m_1 \ge m_2 \ge 2$, excluding  
$\lambda=(2, 2)$ and $(3, 3)$, $V_{\lambda}$ is not thick. 
As mentioned in Section~\ref{section:thickrepresentations},  
$n=m_1+m_2$.  
Note that 
\begin{eqnarray}
d(m_1, m_2)=\dim_{{\Bbb K}} V_{(m_1, m_2)} = \dfrac{(m_1+m_2)!(m_1-m_2+1)}{(m_1+1)!m_2!} = \sharp Y^{\rm st}_{(m_1, m_2)} \label{eq:dimSm1m2} 
\end{eqnarray} 
by Theorem~\ref{th:basisofSpecht} and Proposition~\ref{prop:formuladim}. 


\begin{definition}\label{def:TijMij}\rm 
For $1\le i\neq  j \le n$, let $M_{i,j}$ be the subspace of $V_{(m_1, m_2)}$ spanned by 
\[
T_{i, j} = \left\{  \begin{array}{c|c} 
{\small\begin{ytableau}
        i & a_3 & a_5 & \cdots & a_{n}\\
        j & a_4 & \cdots 
\end{ytableau}} \in Y_{(m_1, m_2)} & \{i, j, a_3, \ldots, a_n \} =\{ 1, 2, \ldots, n \}  
\end{array} 
\right\}.  
\] 
Set $W_1=M_{1, 2}$. For $\sigma \in S_n$, $\sigma(W_1)=M_{\sigma(1), \sigma(2)}$. 
Note that $\dim_{{\Bbb K}} W_1=\sharp Y^{\rm st}_{(m_1-1, m_2-1)} = d(m_1-1, m_2-1)$.  
\end{definition} 

\bigskip 

Suppose that a subspace $W_2$ of $V_{(m_1, m_2)}$ satisfies the following property:
\begin{align}
\mbox{For any $1 \le i \neq j \le n$}, 
\mbox{ there exists a (not necessarily standard) tableau }  \label{property:W2}   \\ 
{\small\begin{ytableau}
        i & a_3 & a_5 & \cdots & a_{n} \\
        j & a_4 & \cdots 
\end{ytableau}} \in W_2. \nonumber 
\end{align}

\begin{proposition}\label{prop:W1W2}
Let $W_1$ and $W_2$ be as above. Then $V_{(m_1, m_2)}$ is not 
$(\dim_{{\Bbb K}} W_1, \dim_{{\Bbb K}} W_2)$-thick. 
If $\dim_{{\Bbb K}} W_1+ \dim_{{\Bbb K}} W_2 \le d(m_1, m_2)$, then 
$V_{(m_1, m_2)}$ is not thick. In particular, Proposition~\ref{prop:main5} holds in this case. 
\end{proposition}

\proof
Since $\sigma(W_1)=M_{\sigma(1), \sigma(2)}$ and $W_2$ has the property (\ref{property:W2}), $\sigma(W_1)\cap W_2\neq 0$ for any $\sigma \in S_n$.  
\qed 

\bigskip 

By Proposition~\ref{prop:W1W2}, it suffices to find a subspace $W_2$ of $V_{(m_1, m_2)}$ 
satisfying (\ref{property:W2}) in order to prove Proposition~\ref{prop:main5}. 
We prove Proposition~\ref{prop:main5} by dividing the argument into several cases. 
 

\bigskip 

First, let us consider the subcase $\lambda=(m_1, 2)$. 
Let $m_1=n-2$ and $m_2=2$. 
By Proposition~\ref{prop:okuyama-omoda}, $V_{(n-2, 2)}$ is not $2$-thick for $n\ge 5$, which 
implies that Proposition~\ref{prop:main5} holds in the subcase $\lambda=(m_1, 2)$ with $m_1\ge 3$. Moreover, we have a more refined result for the 
$(i, j)$-thickness of $V_{(n-2, 2)}$ for $n\ge 7$. 
 
\begin{example}\label{ex:n-2,2}\rm
Note that 
\begin{eqnarray*}
\dim_{{\Bbb K}} V_{(n-2, 2)}=\dfrac{n!(n-3)}{(n-1)!2!} = \dfrac{n(n-3)}{2},  \\
\dim_{{\Bbb K}} W_1 = \sharp Y^{\rm st}_{(n-3, 1)} = d(n-3, 1) = n-3.  
\end{eqnarray*} 
Let $W_2$ be the subspace of $V$ spanned by 
\begin{align*} 
\ytableausetup{boxsize=2.2em}
T_{1,2} \cup \left\{  \begin{array}{c|c} 
       {\small\begin{ytableau}
        1 & 2 & \scriptstyle a_5(i) & \cdots & \scriptstyle a_{n}(i)\\
        i & i+1  
\end{ytableau}} \in  Y_{(m_1, m_2)} & i=3, 4, \ldots, n-1 
\end{array} \right\}  \\
\cup \left\{ {\small\begin{ytableau}
        1 & 2 & \scriptstyle a_5(n) & \cdots & \scriptstyle a_{n}(n)\\
        n & 3  
\end{ytableau}} \in  Y_{(m_1, m_2)} \right\},   
\end{align*} 
where we choose a tuple $(a_5(i), \ldots, a_n(i))$ for $i=3, \ldots, n$.  
Then $W_2$ has the property (\ref{property:W2}) and $\dim_{{\Bbb K}} W_2 \le \dim_{{\Bbb K}} W_1+(n-2)=2n-5$. 
By direct calculation, 
$\dim_{{\Bbb K}} W_1+\dim_{{\Bbb K}} W_2 \le 3n-8 \le \dfrac{n(n-3)}{2} =\dim_{{\Bbb K}} V_{(n-2, 2)}$ for 
$n\ge 7$. Hence, $V_{(n-2, 2)}$ is not $(n-3, 2n-5)$-thick for $n\ge 7$ by Remark~\ref{rem:thickdiagram}, since $V_{(n-2, 2)}$ is not $(\dim_{{\Bbb K}} W_1, \dim_{{\Bbb K}} W_2)$-thick.

\end{example} 

By Example~\ref{ex:n-2,2}, we have the following proposition: 

\begin{proposition}
For $n\ge 7$, $V_{(n-2, 2)}$ is not $(n-3, 2n-5)$-thick. 
In particular, it is not $m$-thick for any $n-3 \le m \le \dfrac{(n-2)(n-5)}{2}$. 
\end{proposition}

\proof
The latter part follows from Proposition~\ref{prop:ijthick}.  
\qed 

\ytableausetup{boxsize=normal}

\bigskip 

In the sequel, we assume that $m_1 \ge m_2 \ge 3$, except for $\lambda=(3, 3)$.  
Then $n=m_1+m_2\ge 7$.

\begin{definition}\label{def:standardwrtS}\rm
A tableau is {\it standard with respect to $S$} (or {\it w.r.t. $S$}) for 
a subset $S \subseteq 
\{1, \ldots, n\}$ if it is filled with the
elements of $S$, each appearing exactly once, and the entries increase
strictly along rows and down columns. 
Set 
\[
T'_{1,2} = \left\{  \begin{array}{c|c} 
{\small\begin{ytableau}
        1 & a_3 & a_5 & \cdots & a_{n}\\
        2 & a_4 & \cdots 
\end{ytableau}} \in Y_{(m_1, m_2)} & {\small\begin{ytableau}
         a_3 & a_5 & \cdots & a_{n}\\
         a_4 & \cdots 
\end{ytableau}} \mbox{ is standard w.r.t. } \{ 3, 4, \ldots, n \}  
\end{array} 
\right\}
\]
and 
\[
T'_{1,3} = \left\{  \begin{array}{c|c} 
{\small\begin{ytableau}
        1 & a_3 & a_5 & \cdots & a_{n}\\
        3 & a_4 & \cdots 
\end{ytableau}} \in  Y_{(m_1, m_2)} & {\small\begin{ytableau}
         a_3 & a_5 & \cdots & a_{n}\\
         a_4 & \cdots 
\end{ytableau}} \mbox{ is standard w.r.t. } \{ 2, 4, \ldots, n \}  
\end{array} 
\right\}. 
\] 
Note that $T'_{1,2} \coprod  T'_{1,3} \subseteq Y^{\rm st}_{(m_1, m_2)}$ and 
that $\sharp T'_{1,2}= \sharp T'_{1,3} = d(m_1-1, m_2-1)$. 
We put $Z_{(m_1, m_2)} = Y^{\rm st}_{(m_1, m_2)}\setminus (T'_{1,2} \coprod  T'_{1,3})$. 
\end{definition}

\begin{remark}\rm 
The subspaces $M_{1, 2}$ and $M_{1, 3}$ of $V$ are generated by $T'_{1, 2}$ and $T'_{1, 3}$, respectively. 
Furthermore, $M_{1, 2} \cap M_{1,3}=0$ and 
$\dim_{{\Bbb K}} M_{1, 2} = \dim_{{\Bbb K}} M_{1, 3} = \sharp T'_{1, 2} = \sharp T'_{1, 3}= d(m_1-1, m_2-1)$.  
\end{remark}

\begin{definition}\label{def:m1m2caseW2}\rm 
Let $W_2$ be the subspace of $V_{(m_1, m_2)}$ spanned by 
\[
\begin{array}{cc} 
\ytableausetup{boxsize=2.2em}
T_{1,2} \cup \left\{  \begin{array}{c|c} 
       {\small \begin{ytableau}
        1 & 2 & \scriptstyle a_5(i) & \cdots & \scriptstyle a_{n}(i)\\
        \scriptstyle 2i-1 & 2i & \cdots  
\end{ytableau}} \in  Y_{(m_1, m_2)} & i=2, 3, \ldots, \dfrac{n}{2}  
\end{array} \right\} & \mbox{ if $n$ is even}, \\ 
\begin{array}{c} T_{1,2} \cup \left\{  \begin{array}{c|c} 
       {\small \begin{ytableau}
        1 & 2 & \scriptstyle a_5(i) & \cdots & \scriptstyle a_{n}(i)\\
        \scriptstyle 2i-1 & 2i & \cdots  
\end{ytableau}} \in  Y_{(m_1, m_2)} & i=2, 3, \ldots, \dfrac{n-1}{2}  
\end{array} \right\} \\ 
\cup \left\{
{\small \ytableausetup{boxsize=3em} \begin{ytableau}
        1 & 2 & \scriptstyle a_5(\frac{n+1}{2}) & \cdots & \scriptstyle a_{n}(\frac{n+1}{2})\\
        \scriptstyle n & b & \cdots  
\end{ytableau}} \in  Y_{(m_1, m_2)} 
\right\} \end{array} 
& \mbox{ if $n$ is odd},   
\end{array} 
\]
where we arbitrarily choose a number $b$ and a tuple $(a_5(i), \ldots, a_n(i))$ for each $i=2, \ldots, \left[\dfrac{n+1}{2}\right]$ to obtain Young tableaux. 
Then $W_2$ has the property (\ref{property:W2}), because 
\[\ytableausetup{boxsize=2.2em} 
{\small\begin{ytableau}
      1 & 2 & \scriptstyle a_5(i) & \cdots \\
      2i &  \scriptstyle  2i-1 & \cdots 
\end{ytableau}} = 
{\small\begin{ytableau}
      1 & 2i & \scriptstyle a_5(i) & \cdots \\
      2 &  \scriptstyle 2i-1 & \cdots 
\end{ytableau}}
+ 
{\small\begin{ytableau}
      1 & 2 & \scriptstyle a_5(i) & \cdots \\
      \scriptstyle 2i-1 & 2i & \cdots 
\end{ytableau}}  \in W_2 
\]
and 
\[\ytableausetup{boxsize=3em}  
{\small\begin{ytableau}
      1 & 2 & \scriptstyle a_5(\frac{n+1}{2}) & \cdots \\
      b &  n & \cdots 
\end{ytableau}} = 
{\small\begin{ytableau}
      1 & b & \scriptstyle a_5(\frac{n+1}{2}) & \cdots \\
      2 & n & \cdots 
\end{ytableau}}
+ 
{\small\begin{ytableau}
      1 & 2 & \scriptstyle a_5(\frac{n+1}{2}) & \cdots \\
      n & b & \cdots 
\end{ytableau}} \in W_2. 
\]
Note that $\dim_{{\Bbb K}} W_2 \le \dim_{{\Bbb K}} W_1+\left[\dfrac{n-1}{2}\right] = \sharp T'_{1,3} +\left[\dfrac{n-1}{2}\right]$.  
\end{definition} 

\bigskip 

\ytableausetup{boxsize=normal}

Choosing $W_2$ as in Definition~\ref{def:m1m2caseW2}, we obtain the following result. 

\begin{proposition}\label{prop:usefulpropnotthick}
Let $m_1\ge m_2\ge 3$ and $n=m_1+m_2\ge 7$.  
Then $V_{(m_1, m_2)}$ is not $(d(m_1-1, m_2-1), d(m_1-1, m_2-1)+\left[\dfrac{n-1}{2}\right])$-thick. 
If 
\begin{align} 
\sharp Z_{(m_1, m_2)} \ge \left[\dfrac{n-1}{2}\right], \label{ineq:Zgen-1/2} 
\end{align}  
then $2d(m_1-1, m_2-1)+ \left[\dfrac{n-1}{2}\right] \le d(m_1, m_2)$ and 
$V_{(m_1, m_2)}$ is not thick. 
\end{proposition}

\proof
Let $W_1$ and $W_2$ be as above. 
The first part follows from the properties of $W_1$ and $W_2$. 
Suppose that $\sharp Z_{(m_1, m_2)} \ge \left[\dfrac{n-1}{2}\right]$. Since 
\begin{align*}
\dim_{{\Bbb K}} W_1+\dim_{{\Bbb K}} W_2 & \le \sharp T'_{1,2} + \sharp T'_{1,3} + \left[\dfrac{n-1}{2}\right]  \\
 & \le \sharp T'_{1,2} + \sharp T'_{1,3} + \sharp Z_{(m_1, m_2)} \\
 & = d(m_1, m_2),   
\end{align*} 
$2d(m_1-1, m_2-1)+ \left[\dfrac{n-1}{2}\right] \le d(m_1, m_2)$ and $V_{(m_1, m_2)}$ is not thick. 
\qed 

\bigskip 

To prove Proposition~\ref{prop:main5} for $\lambda=(m_1, m_2)$ with $m_1\ge m_2\ge 3$ and $n =m_1+m_2\ge 7$, we only need to prove 
$\sharp Z_{(m_1, m_2)} \ge \left[\dfrac{n-1}{2}\right]$ by Proposition~\ref{prop:usefulpropnotthick}. 

\begin{proposition}
Let $m_1>m_2=3$ and $n=m_1+m_2\ge 7$. 
Then $\sharp Z_{(m_1, m_2)} \ge \left[\dfrac{n-1}{2}\right]$. 
In particular, $V_{(n-3, 3)}$ is not thick for $n\ge 7$. 
\end{proposition}

\proof 
Recall $Z_{(m_1, m_2)} = Y^{\rm st}_{(m_1, m_2)} \setminus (T'_{1,2} \coprod  T'_{1,3})$. 
Since 
\begin{eqnarray*} 
Z_{(m_1, m_2)} & \supset & 
\left\{
\begin{array}{c|c}
 {\small\begin{ytableau}
        1 & 2 & 3 &  \cdots  \\
        i & j & k   
\end{ytableau}} & 4\le i<j<k \le n  
\end{array}
\right\},  
\end{eqnarray*} 
$\sharp Z_{(m_1, m_2)} \ge \binom{n-3}{3} = (n-3)\dfrac{(n-4)\cdot (n-5)}{3\cdot 2}\ge n-3 \ge \left[\dfrac{n-1}{2}\right]$ for $n\ge 7$. 
The statement follows from Proposition~\ref{prop:usefulpropnotthick}. 
\qed 

\begin{proposition}
Let $m_1\ge m_2=4$ and $n=m_1+m_2\ge 8$. 
Then $\sharp Z_{(m_1, m_2)} \ge \left[\dfrac{n-1}{2}\right]$. 
In particular, $V_{(n-4, 4)}$ is not thick for $n\ge 8$. 
\end{proposition}

\proof
When $n=8$ or $9$, 
\begin{eqnarray*} 
Z_{(4, 4)} & = & 
\left\{
 {\small\begin{ytableau}
        1 & 2 & 3 & 4   \\
        5 & 6 & 7 & 8    
\end{ytableau}}, \quad    
 {\small\begin{ytableau}
        1 & 2 & 3 & 5   \\
        4 & 6 & 7 & 8    
\end{ytableau}}, \quad    
 {\small\begin{ytableau}
        1 & 2 & 3 & 6   \\
        4 & 5 & 7 & 8    
\end{ytableau}}, \quad    
 {\small\begin{ytableau}
        1 & 2 & 3 & 7   \\
        4 & 5 & 6 & 8    
\end{ytableau}}
\right\},  \\ 
Z_{(5, 4)} & \supset & 
\left\{
\begin{array}{c|c}
 {\small\begin{ytableau}
        1 & 2 & 3 & 4 & m   \\
        i & j & k & l 
\end{ytableau}} & 
\begin{array}{c} 
5\le i<j<k<l \le 9, \\ 
\{ i, j, k, l, m\} = \{ 5, 6, 7, 8, 9 \}   
\end{array} 
\end{array}
\right\}.  
\end{eqnarray*} 
Then 
\begin{align*} 
\sharp Z_{(4, 4)} & = 4 \ge \left[\dfrac{8-1}{2}\right] \text{ and }\; \sharp Z_{(5, 4)} \ge 5 \ge \left[\dfrac{9-1}{2}\right].  
\end{align*}
If $n\ge 10$, then 
\begin{align*}
Z_{(n-4, 4)} & \supset  \left\{
\begin{array}{c|c}
 {\small\begin{ytableau}
        1 & 2 & 3 & 4 & \cdots   \\
        i & j & k & l 
\end{ytableau}} & 
\begin{array}{c} 
5\le i<j<k<l \le n, \\ 
\{ i, j, k, l, \cdots \} = \{ 5, \ldots, n \}   
\end{array} 
\end{array}
\right\} 
\end{align*}
and 
\[
\sharp Z_{(n-4, 4)} \ge \binom{n-4}{4} \ge \binom{n-4}{2} = \dfrac{(n-4)(n-5)}{2} \ge \dfrac{5(n-4)}{2} \ge \left[\dfrac{n-1}{2}\right] 
\]     
by Lemma~\ref{lemma:binomineq-1}. 
This completes the proof.  
\qed 

\begin{proposition}\label{prop:inequalityofg}
If $m_1 \ge m_2\ge 5$, then $d(m_1, m_2) \ge 3 d(m_1-1, m_2-1)$ and $d(m_1-1, m_2-1)\ge m_1+m_2-3$.  
\end{proposition}

\proof
Let $n=m_1+m_2$ and $m_1 \ge m_2\ge 5$.  
By (\ref{eq:dimSm1m2}), 
\begin{align*}
\dfrac{d(m_1, m_2)}{d(m_1-1, m_2-1)} = \dfrac{(m_1+m_2)(m_1+m_2-1)}{(m_1+1)m_2} 
= \dfrac{n(n-1)}{(n-m_2+1)m_2}.  
\end{align*} 
For $x=n\ge 2m_2$, we define 
\begin{align*}
g(x) = \dfrac{x(x-1)}{(x-m_2+1)m_2}.
\end{align*} 
Since 
\begin{align*}
(\log g(x))' = \dfrac{g'(x)}{g(x)} = \dfrac{1}{x}+\dfrac{1}{x-1}-\dfrac{1}{x-m_2+1} 
= \dfrac{x^2-2(m_2-1)x+(m_2-1)}{x(x-1)(x-m_2+1)},  
\end{align*} 
$g'(x)>0$ for $x\ge 2m_2$. 
For $x \ge 2m_2$, we have 
\begin{align*}
g(x) \ge g(2m_2) = \dfrac{2m_2(2m_2-1)}{(2m_2-m_2+1)m_2} = \dfrac{4m_2-2}{m_2+1} 
=4 -\dfrac{6}{m_2+1} \ge 4 -\dfrac{6}{5+1} = 3
\end{align*}
for $m_2 \ge 5$, which implies that $d(m_1, m_2) \ge 3 d(m_1-1, m_2-1)$.  
Using (\ref{eq:dimSm1m2}) again, we obtain  
\begin{align*}
d(m_1-1, m_2-1) &= \dfrac{(m_1+m_2-2)!(m_1-m_2+1)}{m_1!(m_2-1)!}  \\ 
&= \binom{m_1+m_2-2}{m_2-2}\cdot\dfrac{m_1-m_2+1}{m_2-1} \\ 
&\ge \binom{m_1+m_2-2}{2}\cdot\dfrac{m_1-m_2+1}{m_2-1} \\ 
&= (m_1+m_2-3)\cdot \dfrac{(m_1-1+m_2-1)}{2(m_2-1)}\cdot(m_1-m_2+1) \\ 
&\ge m_1+m_2-3
\end{align*} 
for $m_1\ge m_2\ge 5$ by Lemma~\ref{lemma:binomineq-1}. This completes the proof. 
\qed 

\begin{proposition}
For $m_1 \ge m_2 \ge 5$, then $\sharp Z_{(m_1, m_2)} \ge \left[\dfrac{n-1}{2}\right]$. 
In particular, $V_{(m_1, m_2)}$ is not thick. 
\end{proposition}

\proof
Note that $\sharp Z_{(m_1, m_2)} = d(m_1, m_2) - 2d(m_1-1, m_2-1)$. By Proposition~\ref{prop:inequalityofg}, 
\begin{align*}
\sharp Z_{(m_1, m_2)} &\ge 3d(m_1-1, m_2-1) - 2d(m_1-1, m_2-1) \\
 & = d(m_1-1, m_2-1) \\ 
 & \ge m_1+m_2-3 \\
  & \ge \left[\dfrac{m_1+m_2-1}{2}\right] = \left[\dfrac{n-1}{2}\right]. 
\end{align*}
Proposition~\ref{prop:usefulpropnotthick} implies that $V_{(m_1, m_2)}$ is not thick.  
\qed 

\bigskip 

In the case $\lambda=(m_1, m_2)$ with $m_1\ge m_2 \ge 3$ with $n=m_1+m_2\ge 7$, 
we have verified that $\sharp Z_{(m_1, m_2)} \ge \left[\dfrac{n-1}{2}\right]$. 
By Proposition~\ref{prop:usefulpropnotthick}, we have: 

\begin{proposition}
For $\lambda=(m_1, m_2)$ with $m_1\ge m_2 \ge 3$ with $n=m_1+m_2\ge 7$, 
$V_{\lambda}$ is not $(d(m_1-1, m_2-1), d(m_1-1, m_2-1)+\left[\dfrac{n-1}{2}\right])$-thick. 
In particular, $V_{\lambda}$ is not $m$-thick for any $d(m_1-1, m_2-1) \le m \le 
d(m_1, m_2)-d(m_1-1, m_2-1)-\left[\dfrac{n-1}{2}\right]$. 
\end{proposition}

\proof 
The latter part follows from Proposition~\ref{prop:ijthick}.  
\qed 

\bigskip 

Summarizing the discussion above, 
Proposition~\ref{prop:main5} has been proved in Case~1. 

\section{The case $\lambda = (m_1, m_2, m_3)$ with $m_1 \ge m_2 \ge m_3 \ge 1$}\label{section:s=3}  

In this section, we prove Proposition~\ref{prop:main5} in Case~2; namely, 
that $V_{(m_1, m_2, m_3)}$ is not thick if 
it is not contained in the list of Theorem~\ref{th:main3}. 
More precisely, if $\lambda=(m_1, m_2, m_3)$ with 
$m_1\ge m_2\ge m_3 \ge 2$, excluding $\lambda=(2, 2, 2)$, or 
$\lambda=(m_1, m_2, 1)$ with $m_1 \ge m_2 \ge 1$, excluding 
$\lambda=(1, 1, 1)$ and $(2, 1, 1)$, then  
$V_{\lambda}$ is not thick. 
As mentioned in Section~\ref{section:thickrepresentations},  
$n=m_1+m_2+m_3$.  

For $n\ge 5$, $V_{(n-2, 1, 1)}$ is not $3$-thick by Proposition~\ref{prop:311and411}, 
which implies that Theorem~\ref{th:main4} and Proposition~\ref{prop:main5} hold.     
Hence, to prove Proposition~\ref{prop:main5}, the subcase $\lambda=(m_1, m_2, 1)$ with $m_1 \ge m_2 \ge 1$, excluding 
$\lambda=(1, 1, 1)$ and $(2, 1, 1)$ can be reduced to the subcase 
$\lambda=(m_1, m_2, 1)$ with $m_1 \ge m_2 \ge 2$.   

In Subsection~\ref{subsection:m1m21}, we prove Proposition~\ref{prop:main5} 
in the case $\lambda=(m_1, m_2, 1)$ with $m_1 \ge m_2 \ge 2$. 
In Subsection~\ref{subsection:m1m22}, we prove Proposition~\ref{prop:main5} 
in the case $\lambda=(m_1, m_2, 2)$ with $m_1 \ge m_2 \ge 2$, excluding the case 
$\lambda=(2, 2, 2)$. 
In Subsection~\ref{subsection:m1m2m3ge3}, we prove Proposition~\ref{prop:main5} 
in the case $\lambda=(m_1, m_2, m_3)$ with $m_1 \ge m_2 \ge m_3 \ge 3$. 

\bigskip 

Before proving each case, we introduce the function $F(n)$. 
Let $X_n = \{ 1, 2, \ldots, n\}$. Let $\binom{X_n}{3}$ be the set of all $3$-element subsets of $X_n$. 
For $F \subseteq \binom{X_n}{3}$, we say that $F$ {\it covers all edges} in $X_n$ 
if for any $i\neq j \in X_n$ there exists 
$f\in F$ such that $i, j \in f$. 
Then we define 
\[
F(n) = \min\left\{ 
\begin{array}{c|c} 
\sharp F & \binom{X_n}{3} \supseteq F \mbox{  covers all edges in $X_n$}
\end{array} \right\}
\]
for $n\ge 3$. 
For example, $F(3)=1, F(4)=3, F(5)=4$. 

\begin{proposition}\label{prop:ineqofF} 
For $k \ge 2$, 
$F(2k+1)\le k^2$ and $F(2k)\le k^2-1$. 
\end{proposition}

\proof
If $k=2$, then the statement is true.  Let us prove the proposition by induction on $k$. 
Assume that $F \subseteq \binom{X_{2k+1}}{3}$ covers all edges in $X_{2k+1}=\{ 1, 2, \ldots, 2k+1 \}$ and 
$\sharp F \le k^2$. 
Then the union of $F$ and the set of $2k+1$ elements 
$\{ 2k+3, 1, 2\}, \{ 2k+3, 3, 4\}, \ldots \{ 2k+3, 2k+1, 2k+2\}, \{ 2k+2, 1, 2\}, \{ 2k+2, 3, 4\}, \ldots, \{ 2k+2, 2k-1, 2k\}$  
covers all edges in $X_{2k+3}=\{ 1, 2, \ldots, 2k+3 \}$. Thereby, we have $F(2k+3) \le k^2+2k+1 =(k+1)^2$. 

Assume that $F \subseteq \binom{X_{2k}}{3}$ covers all edges in $X_{2k}=\{ 1, 2, \ldots, 2k \}$ and $\sharp F \le k^2-1$. 
Then the union of $F$ and the set of $2k+1$ elements 
$\{ 2k+2, 1, 2\}, \{ 2k+2, 3, 4\}, \ldots, \{ 2k+2, 2k-1, 2k\}, 
\{ 2k+2, 2k+1, 2k\}, \{ 2k+1, 1, 2\}, \{ 2k+1, 3, 4\}, \ldots, \{ 2k+1, 2k-1, 2k\}$    
covers all edges in $X_{2k+2}=\{ 1, 2, \ldots, 2k+2 \}$. Thereby, we have $F(2k+2) \le k^2-1+2k+1 =(k+1)^2-1$. 
Hence, we have proved the statement. 
\qed 

\begin{corollary}\label{cor:ineqofF}
For $n\ge 3$, $F(n) \le \dfrac{n^2-4}{4}$. 
\end{corollary}

\proof
If $n=3$, then $F(3)=1 \le \dfrac{3^2-4}{4}$. 
If $n$ is odd with $n\ge 5$, then $F(n)\le \left(\dfrac{n-1}{2}\right)^2 \le \dfrac{n^2-4}{4}$ 
by Proposition~\ref{prop:ineqofF}. 
If $n$ is even with $n\ge 4$, then $F(n)\le \left(\dfrac{n}{2}\right)^2 -1 = \dfrac{n^2-4}{4}$ 
by Proposition~\ref{prop:ineqofF}.  This completes the proof. 
\qed

\subsection{The case $\lambda = (m_1, m_2, 1)$ with $m_1\ge m_2\ge 2$}\label{subsection:m1m21}  
In this subsection, we consider the subcase 
$\lambda=(m_1, m_2, 1)$ with $m_1 \ge m_2 \ge 2$. In this subcase, $n=m_1+m_2+1\ge 5$. 
%
%
By Proposition~\ref{prop:formuladim}, we have 
\[
d(m_1, m_2, 1) = \dim_{{\Bbb K}} V_{(m_1, m_2, 1)} = \binom{m_1+m_2+1}{m_2-1} \cdot \dfrac{m_1+1}{m_2+1}\cdot (m_1-m_2+1). 
\]
Let 
\[ 
W_1 = \left\langle 
\begin{array}{c|c} {\small\begin{ytableau}
        1 & \cdots   \\
        2 & \cdots \\
        i  
\end{ytableau}} \in Y_{(m_1, m_2, 1)} & i=3, 4, \ldots, n \end{array}  
\right\rangle.
\]   
We see that 
\begin{align*} 
\dim_{{\Bbb K}} W_1 & =(n-2) \sharp Y^{\rm st}_{(m_1-1, m_2-1)} = (n-2)\dim_{{\Bbb K}} V_{(m_1-1, m_2-1)} \\
& =(m_1+m_2-1)\binom{m_1+m_2-2}{m_2-1}\dfrac{m_1-m_2+1}{m_1}. 
\end{align*} 
Setting 
$g(m_1, m_2) = \dfrac{d(m_1, m_2, 1)}{\dim_{{\Bbb K}} W_1}$, we have  
\begin{align*}
g(m_1, m_2) 
 & = \dfrac{(m_1+m_2+1)!}{(m_2-1)!(m_1+2)!}\cdot \dfrac{m_1+1}{m_2+1} \cdot  \dfrac{m_1(m_2-1)!(m_1-1)!}{(m_1+m_2-2)!(m_1+m_2-1)} \\
 & = \dfrac{(m_1+m_2+1)(m_1+m_2)}{(m_1+2)(m_2+1)}. 
\end{align*} 

Putting $m_1=m_2+k$ with $k\ge 0$, we obtain 
\begin{align*}
g(m_1, m_2) & = g(m_2+k, m_2) = \dfrac{(2m_2+k+1)(2m_2+k)}{(m_2+k+2)(m_2+1)}. 
\end{align*}  

\begin{lemma}\label{lemma:ineqofg}
Let $k\ge 0$ and $m_2\ge 2$. Then $g(m_1, m_2)\ge 2$ if and only if either 
$k\ge 1$ and $m_2\ge 2$, or $k=0$ and $m_2\ge 3$. 
\end{lemma} 

\proof 
By direct calculation, 
\begin{align*}
\dfrac{(2m_2+k+1)(2m_2+k)}{(m_2+k+2)(m_2+1)} \ge 2 
&\Longleftrightarrow m_2^2+(k-2)m_2+\dfrac{k^2-k-4}{2} \ge 0 \\
&\Longleftrightarrow \left(m_2+\dfrac{k-2}{2}\right)^2 +\dfrac{k^2+2k-12}{4} \ge 0. 
\end{align*}
It is straightforward to verify the statement. 
\qed 

\begin{proposition}\label{prop:dimW1geFn}
For $n=m_1+m_2+1 = 2m_2+k+1$, 
we have $\dim_{{\Bbb K}} W_1 \ge F(n)$ if either $k\ge 1$ and $m_2\ge 2$, or 
$k=0$ and $m_2\ge 3$.  
\end{proposition}

\proof
When $k=0$ and $m_2=3$, we can verify that $n=7$ and $\dim_{{\Bbb K}} W_1=10 \ge 3^2 \ge F(7)$ by 
Proposition~\ref{prop:ineqofF}.  Thereby, let us show the statement if $k\ge 1$ and $m_2\ge 2$, or if 
$k=0$ and $m_2\ge 4$.  
For $m_1=m_2+k$ with $k\ge 0$, set 
\[
h(m_2, k) =\dfrac{\dim_{{\Bbb K}} W_1}{F(n)}  
\]
and 
\[
\underline{h}(m_2, k) =\dfrac{\dim_{{\Bbb K}} W_1}{\frac{n^2-4}{4}} = \dfrac{4\dim_{{\Bbb K}} W_1}{n^2-4}.      
\]
By Corollary~\ref{cor:ineqofF}, $h(m_2, k) \ge \underline{h}(m_2, k)$. 
It suffices to show that $\underline{h}(m_2, k)\ge 1$ if $k\ge 1$ and $m_2\ge 2$, or if 
$k=0$ and $m_2\ge 4$. 
By direct calculation, 
\begin{align}
\underline{h}(m_2, k) &= 4(m_1+m_2-1)\cdot \binom{m_1+m_2-2}{m_2-1}\cdot \dfrac{m_1-m_2+1}{m_1}\cdot \dfrac{1}{n^2-4} \nonumber \\ 
& = 4(2m_2+k-1)\cdot \dfrac{(2m_2+k-2)!}{(m_2-1)!(m_2+k-1)!} \cdot \dfrac{k+1}{m_2+k}\cdot \dfrac{1}{(2m_2+k+1)^2-4} \nonumber \\
& = 4\cdot \dfrac{(2m_2+k-2)!}{(m_2-1)!(m_2+k)!} \cdot \dfrac{k+1}{2m_2+k+3}.   
\label{eq:dimW1Fn} 
\end{align} 

When $k=0$, we have 
\begin{align*}
\underline{h}(m_2, 0) &= \dfrac{4 (2m_2-2)!}{(m_2-1)!m_2!(2m_2+3)}.  
\end{align*}
Since $\underline{h}(4, 0)=\dfrac{4\cdot 6!}{3!4!\cdot 11}=\dfrac{20}{11}>1$ and  
\begin{align*}
\dfrac{\underline{h}(m_2+1, 0)}{\underline{h}(m_2, 0)} 
& =  \dfrac{4 (2m_2)!}{m_2!(m_2+1)!(2m_2+5)}\cdot 
 \dfrac{(m_2-1)!m_2!(2m_2+3)}{4(2m_2-2)!}  \\
 & = \dfrac{2(2m_2-1)(2m_2+3)}{(m_2+1)(2m_2+5)} \ge 1 
\end{align*}
for $m_2\ge 4$, we see that $\underline{h}(m_2, k)\ge 1$ for $k=0$ and $m_2\ge 4$. 

Let us consider the case $k\ge 1$. By (\ref{eq:dimW1Fn}), 
\begin{align*}
\underline{h}(2, k) &= 4\cdot \dfrac{(k+2)!}{1!(k+2)!} \cdot \dfrac{k+1}{k+7} \\ 
 & = \dfrac{4(k+1)}{k+7} \ge 1. 
\end{align*}
Using (\ref{eq:dimW1Fn}) again, we see that 
\begin{align*}
& \dfrac{\underline{h}(m_2+1, k)}{\underline{h}(m_2, k)} \\
& =  \dfrac{(2m_2+k)!}{m_2!(m_2+k+1)!} \cdot \dfrac{k+1}{2m_2+k+5}\cdot \dfrac{(m_2-1)!(m_2+k)!}{(2m_2+k-2)!} \cdot \dfrac{2m_2+k+3}{k+1}  \\
& = \dfrac{(2m_2+k)(2m_2+k-1)}{m_2(m_2+k+1)}\cdot \dfrac{2m_2+k+3}{2m_2+k+5} \\
& = \dfrac{2m_2+k}{m_2+k+1}\cdot \dfrac{4m_2^2+(4k+4)m_2+k^2+2k-3}{2m_2^2+(k+5)m_2} \ge 1,  
\end{align*}
since $2m_2+k \ge m_2+k+1$ and $4m_2^2+(4k+4)m_2+k^2+2k-3 \ge 2m_2^2+(k+5)m_2$ for $m_2\ge 2$ and $k\ge 1$. 
Hence, $\underline{h}(m_2, k)\ge 1$ if $k\ge 1$ and $m_2\ge 2$.  
This completes the proof. 
\qed 

\begin{corollary}\label{cor:ineqofSdimW1Fn} 
If $k\ge 1$ and $m_2\ge 2$, or if $k=0$ and $m_2\ge 3$, then 
$\dim_{{\Bbb K}} V_{(m_1, m_2, 1)} \ge \dim_{{\Bbb K}} W_1 +F(m_1+m_2+1)$.   
\end{corollary}

\proof
The statement follows from Lemma~\ref{lemma:ineqofg} and Proposition~\ref{prop:dimW1geFn}. 
\qed 

\bigskip 

Let us define a subspace $W_2$ of $V_{(m_1, m_2, 1)}$ for $m_1\ge m_2\ge 2$. 
For $n=m_1+m_2+1 \ge 5$, 
take a subset $F_n \subseteq \binom{X_n}{3}$ such that $\sharp F_n = F(n)$ and 
$F_n$ covers all edges in $X_n$. For each $f \in F_n$, choose a (not necessarily standard) Young tableau $y_f$ of form 
\[
 {\small\begin{ytableau}
   i & \cdots   \\
   j & \cdots \\
   k 
\end{ytableau}}\; ,  
\] 
where $f=\{ i, j, k \}$. 
Set $W_2 = \langle y_{f} \mid f \in F_n\rangle$. Then $\dim_{{\Bbb K}} W_2 \le F(n)$.
For $\sigma \in S_n$, choose $f \in Y_n$ such that 
$\{\sigma(1), \sigma(2)\} \subset f$. Since $y_f \in \sigma(W_1)$, 
$\sigma(W_1)\cap W_2 \neq 0$.

\begin{theorem}
If $m_1\ge m_2\ge 2$, then $V_{(m_1, m_2, 1)}$ is not thick. In particular, 
Theorem~\ref{th:main4} and Proposition~\ref{prop:main5} hold in this case. 
\end{theorem}

\proof
By Proposition~\ref{prop:221}, $V_{(2, 2, 1)}$ is not thick. 
It suffices to prove the statement for $k\ge 1$ and $m_2\ge 2$, and for $k=0$ and $m_2\ge 3$. 
Let $W_1$ and $W_2$ be as above. 
Corollary~\ref{cor:ineqofSdimW1Fn} implies that 
$\dim_{{\Bbb K}} V_{(m_1, m_2, 1)} \ge \dim_{{\Bbb K}} W_1+\dim_{{\Bbb K}} W_2$.  Hence, $V_{(m_1, m_2, 1)}$ is not thick. 
\qed 

\bigskip 

By the discussion above, we have:

\begin{proposition}
In the case $\lambda=(m_1, m_2, 1)$ with $m_1\ge m_2\ge 2$, except for $\lambda=(2, 2, 1)$, 
$V_{\lambda}$ is not $((n-2)d(m_1-1, m_2-1), F(n))$-thick, where $n=m_1+m_2+1$.  
In particular, it is not $m$-thick for any $F(n) \le m \le d(m_1, m_2, 1)-(n-2)d(m_1-1, m_2-1)$. 
\end{proposition}

\proof 
The latter part follows from Proposition~\ref{prop:ijthick}.  
\qed

\subsection{The case $\lambda = (m_1, m_2, 2)$ with $m_1\ge m_2\ge 2$}\label{subsection:m1m22}   

In this subsection, we consider the subcase 
$\lambda=(m_1, m_2, 2)$ with $m_1 \ge m_2 \ge 2$, excluding $\lambda=(2, 2, 2)$. 
In this subcase, $n=m_1+m_2+2 \ge 7$. 


Let 
\[ 
W_1 = \left\langle 
\begin{array}{c|c} {\small\begin{ytableau}
        1 & \cdots   \\
        2 & \cdots \\
        i  & \ast 
\end{ytableau}} \in Y_{(m_1, m_2, 2)} & i=3, 4, \ldots, n \end{array}  
\right\rangle.
\]   
By the Garnir relation (Theorem~\ref{th:Garnir's relation}), we see that  
\[ 
W_1 = \left\langle 
\begin{array}{c|c} {\small\begin{ytableau}
        1 & \cdots   \\
        2 & \cdots \\
        i  & \ast 
\end{ytableau}} \in Y_{(m_1, m_2, 2)} & i=3, 4, \ldots, n-1 \end{array}  
\right\rangle. 
\]   
Hence, we obtain 
\begin{align*}
\dim_{{\Bbb K}} W_1 & \le (n-3)\sharp Y^{\rm st}_{(m_1-1, m_2-1, 1)} = (n-3)\dim_{{\Bbb K}} V_{(m_1-1, m_2-1, 1)} \\
 & =(m_1+m_2-1)\cdot \dfrac{(m_1+m_2-1)!(m_1-m_2+1)m_1(m_2-1)}{(m_1+1)!m_2!}. 
\end{align*} 

Let us define a subspace $W_2$ of $V_{(m_1, m_2, 2)}$ for $m_1\ge m_2\ge 2$. 
For $n=m_1+m_2+2 \ge 7$, 
take a subset $F_n \subseteq \binom{X_n}{3}$ such that $\sharp F_n = F(n)$ and 
$F_n$ covers all edges in $X_n$. For each $f \in F_n$, choose a (not necessarily standard) Young tableau $y_f$ of form 
\[
 {\small\begin{ytableau}
   i & \cdots   \\
   j & \cdots \\
   k & \ast 
\end{ytableau}}\; ,  
\] 
where $f=\{ i, j, k \}$. 
Set $W_2 = \langle y_{f} \mid f \in F_n\rangle$. Then $\dim_{{\Bbb K}} W_2 \le F(n)$.
For $\sigma \in S_n$, choose $f \in Y_n$ such that 
$\{\sigma(1), \sigma(2)\} \subset f$. Since $y_f \in \sigma(W_1)$, 
$\sigma(W_1)\cap W_2 \neq 0$. 
To prove that $V_{(m_1, m_2, 2)}$ is not thick, we only need to 
show that 
$\dim_{{\Bbb K}} W_1+F(n) \le \dim_{{\Bbb K}} V_{(m_1, m_2, 2)}$. 

\bigskip 

Let us prove the claim that $\dim_{{\Bbb K}} W_1 + F(n) \le \dim_{{\Bbb K}} V_{(m_1, m_2, 2)}$.  
Putting $m_1=m_2+k$ with $k\ge 0$, we have 
\begin{align*}
n & = 2m_2+k+2 \ge 7, \\
d(m_1, m_2, 2) & = \dim_{{\Bbb K}} V_{(m_1, m_2, 2)} = \dfrac{(2m_2+k+2)!(k+1)(m_2+k)(m_2-1)}{2 (m_2+k+2)!(m_2+1)!}, \\
\dim_{{\Bbb K}} W_1 &\le (2m_2+k-1)\cdot \dfrac{(2m_2+k-1)!(k+1)(m_2+k)(m_2-1)}{(m_2+k+1)!m_2!}, \\ 
\mbox{~and~} F(n) & \le \dfrac{n^2-4}{4} = \dfrac{(2m_2+k)(2m_2+k+4)}{4}. 
\end{align*}

Let us discuss when $\dim_{{\Bbb K}} W_1 + F(n) \le d(m_1, m_2, 2)$ holds.  
By direct calculation, 
\begin{align*}
\dfrac{\dim_{{\Bbb K}} W_1}{d(m_1, m_2, 2)} & \le \dfrac{2(2m_2+k-1)(m_2+k+2)(m_2+1)}{(2m_2+k+2)(2m_2+k+1)(2m_2+k)}. 
\end{align*} 
By Lemma~\ref{lemma:binomineq-1}, 
\begin{align*}
d(m_1, m_2, 2) & =  \binom{2m_2+k+2}{m_2} \cdot \dfrac{(k+1)(m_2+k)(m_2-1)}{2(m_2+1)} \\ 
 & \ge \binom{2m_2+k+2}{2} \cdot \dfrac{(k+1)(m_2+k)(m_2-1)}{2(m_2+1)} \\ 
 & = \dfrac{(2m_2+k+2)(2m_2+k+1)(k+1)(m_2+k)(m_2-1)}{4(m_2+1)}, 
\end{align*} 
and hence 
\begin{align*}
\dfrac{F(n)}{d(m_1, m_2, 2)} & \le \dfrac{(2m_2+k)(2m_2+k+4)(m_2+1)}{(2m_2+k+2)(2m_2+k+1)(k+1)(m_2+k)(m_2-1)}. 
\end{align*} 
Thereby, we have 
\begin{align}
& \dfrac{\dim_{{\Bbb K}} W_1+F(n)}{d(m_1, m_2, 2)} \label{eq:dfracW1+Fn/S}  \\ 
& \le \dfrac{(m_2+1)}{(2m_2+k+2)(2m_2+k+1)} \nonumber \\ 
& \times \left\{  \dfrac{2(2m_2+k-1)(m_2+k+2)}{2m_2+k} + \dfrac{(2m_2+k)(2m_2+k+4)}{(k+1)(m_2+k)(m_2-1)}  
\right\}. \nonumber 
\end{align} 

For $k=0$, we obtain 
\begin{align*}
\dfrac{\dim_{{\Bbb K}} W_1+F(n)}{d(m_1, m_2, 2)}  & \le \dfrac{m_2+1}{(2m_2+2)(2m_2+1)} \cdot \left\{ 
\dfrac{2(2m_2-1)(m_2+2)}{2m_2} + \dfrac{2m_2(2m_2+4)}{m_2(m_2-1)} \right\} \\
& = \dfrac{(m_2+2)(2m_2^2+m_2+1)}{2m_2(m_2-1)(2m_2+1)} 
\end{align*}
by  (\ref{eq:dfracW1+Fn/S}).  It is straightforward to see that 
$\dfrac{\dim_{{\Bbb K}} W_1+F(n)}{d(m_1, m_2, 2)}\le 1$ for $m_2\ge 5$, since 
$(m_2+2)(2m_2^2+m_2+1) \le 2m_2(m_2-1)(2m_2+1)$ for $m_2\ge 5$. 


For $k=1$, we obtain 
\begin{align*}
\dfrac{\dim_{{\Bbb K}} W_1+F(n)}{d(m_1, m_2, 2)}  & \le 
\dfrac{8m_2^4+32m_2^3+20m_2^2-2m_2+5}{4(2m_2+3)(2m_2+1)(m_2+1)(m_2-1)} 
\end{align*}
by (\ref{eq:dfracW1+Fn/S}).  It is straightforward to see that 
$\dfrac{\dim_{{\Bbb K}} W_1+F(n)}{d(m_1, m_2, 2)}\le 1$ for $m_2\ge 3$, since 
$8m_2^4+32m_2^3+20m_2^2-2m_2+5 \le 4(2m_2+3)(2m_2+1)(m_2+1)(m_2-1)$ for $m_2\ge 3$. 

Let us consider the case $k\ge 2$. 
By (\ref{eq:dfracW1+Fn/S}), we obtain 
\begin{align}
& \dfrac{\dim_{{\Bbb K}} W_1+F(n)}{d(m_1, m_2, 2)} \label{eq:dimSW1Fkge2} \\
& \le 
 \dfrac{(m_2+1)}{(2m_2+k+2)(2m_2+k+1)} \nonumber \\ 
& \times \left\{  \dfrac{2(2m_2+k-1)(m_2+k+2)}{2m_2+k} + \dfrac{(2m_2+k+2)^2}{(k+1)(m_2+k)(m_2-1)}  
\right\} \nonumber \\ 
& = \dfrac{2(2m_2+k-1)(m_2+k+2)(m_2+1)}{(2m_2+k)(2m_2+k+1)(2m_2+k+2)} 
+ \dfrac{(2m_2+k+2)(m_2+1)}{(k+1)(m_2+k)(m_2-1)(2m_2+k+1)} 
\nonumber \\
& \le \dfrac{2(m_2+k+2)(m_2+1)}{(2m_2+k+1)(2m_2+k+2)} 
+ \dfrac{1}{(k+1)(m_2-1)}, \nonumber  
\end{align} 
since $(2m_2+k)(2m_2+k+4)\le (2m_2+k+2)^2$, $m_2+1 \le \dfrac{1}{2}(2m_2+k+1)$, and 
$2m_2+k+2 \le 2(m_2+k)$  
if $k \ge 2$. By direct calculation, we see that 
\begin{align*}
\dfrac{2(m_2+k+2)(m_2+1)}{(2m_2+k+1)(2m_2+k+2)} & \le \dfrac{2}{3}  \quad 
\text{~and~} \quad \dfrac{1}{(k+1)(m_2-1)} \le \dfrac{1}{3}
\end{align*}
for $k\ge 2$ and $m_2\ge 2$. Thereby, (\ref{eq:dimSW1Fkge2}) implies that 
$\dfrac{\dim_{{\Bbb K}} W_1+F(n)}{d(m_1, m_2, 2)} \le 1$. 

\bigskip 

The only remaining cases are  
$(k, m_2)=(0, 3), (0, 4)$, and $(1, 2)$.  
These cases correspond to $(m_1, m_2, m_3)=(3, 3, 2), (4, 4, 2)$, and $(3, 2, 2)$, 
respectively. 
These cases are listed in the following table. 

\begin{center}
\begin{tabular}{|c|c|c|c|c|}
\hline 
$(m_1, m_2, m_3)$ & $\dim_{{\Bbb K}} V_{(m_1, m_2, m_3)}$ & $\dim_{{\Bbb K}} W_1$ & $F(n)$ & $\dim_{{\Bbb K}} W_1+F(n) \le \dim_{{\Bbb K}} V_{(m_1, m_2, m_3)}$ \\
\hline 
$(3, 3, 2)$ & $42$ & $\le 25$ & $\le 15$ & \text{yes} \\
\hline 
$(4, 4, 2)$ & $252$ & $\le 147$ & $\le 24$ & \text{yes} \\
\hline 
$(3, 2, 2)$ & $21$ & $\le 12$ & $\le 9$ & \text{yes} \\
\hline 
\end{tabular}
\end{center} 
Hence, they are not thick. 

Summarizing the discussion above, we have

\begin{theorem}
If $m_1\ge m_2\ge 2$, then $V_{(m_1, m_2, 2)}$ is not thick except for $(m_1, m_2, m_3)=(2, 2, 2)$.  In particular, 
Theorem~\ref{th:main4} and Proposition~\ref{prop:main5} hold in this case. 
\end{theorem}

\bigskip 

To obtain a more refined result, we prove the following inequality. 

\begin{proposition}\label{prop:m1m22inequality}
In the case $m_1\ge m_2 \ge 2$, excluding $(m_1, m_2) \neq (2, 2)$, 
\[
(m_1+m_2-1) d(m_1-1, m_2-1, 1) \ge F(m_1+m_2+2). 
\]
\end{proposition}

\proof 
Let $k=m_1-m_2\ge 0$. 
Since $F(m_1+m_2+2) = F(2m_2+k+2) \le \dfrac{(2m_2+k)(2m_2+k+4)}{4}$ by Corollary~\ref{cor:ineqofF}, it suffices to show that 
\begin{align*}
I & = \dfrac{\frac{(2m_2+k)(2m_2+k+4)}{4}}{(2m_2+k-1)d(m_2+k-1, m_2-1, 1)}  \\ 
& =  \dfrac{(2m_2+k)(2m_2+k+4)}{4}\cdot \dfrac{(m_2+k+1)!m_2!}{(2m_2+k-1)(2m_2+k-1)!(k+1)(m_2+k)(m_2-1)} \\ 
& \le 1. 
\end{align*}
In the case $m_2\ge 2$ and $k\ge 0$, excluding $(m_2, k)=(2, 0)$, the inequality 
$(2m_2+k)(2m_2+k+4) \le 4(m_2+k)(2m_2+k-1)$ holds. 
Hence, using Lemma~\ref{lemma:binomineq-1}, we obtain 
\begin{align*}
 I & \le  \dfrac{4(m_2+k)(2m_2+k-1)}{4}\cdot \dfrac{(m_2+k+1)!m_2!}{(2m_2+k-1)(2m_2+k-1)!(k+1)(m_2+k)(m_2-1)}  \\
 & =  \dfrac{(m_2+k+1)!m_2!}{(2m_2+k-1)!(k+1)(m_2-1)} = \dfrac{(2m_2+k)(2m_2+k+1)}{\binom{2m_2+k+1}{m_2}(k+1)(m_2-1)} \\ 
 & \le \dfrac{(2m_2+k)(2m_2+k+1)}{\binom{2m_2+k+1}{2}(k+1)(m_2-1)} 
  = \dfrac{2}{(k+1)(m_2-1)} 
  \le 1, 
\end{align*}
if $m_2\ge 2$ and $k\ge 0$ except for $(m_2, k)=(2, 0)$. Thus, we have proved the inequality. 
\qed 

\bigskip 

Combining the discussion above with Proposition~\ref{prop:m1m22inequality}, 
we obtain the following proposition. 

\begin{proposition}
In the case $\lambda=(m_1, m_2, 2)$ with $m_1\ge m_2\ge 2$, except for $\lambda=(2, 2, 2)$, 
$V_{\lambda}$ is not $((n-3)d(m_1-1, m_2-1, 1), F(n))$-thick, where $n=m_1+m_2+2$.   
In particular, it is not $m$-thick for any 
$F(n) \le m \le d(m_1, m_2, 2)-(n-3)d(m_1-1, m_2-1, 1)$. 
\end{proposition}

\proof 
The latter part follows from Proposition~\ref{prop:ijthick}.  
\qed

\subsection{The case: $m_1 \ge m_2\ge m_3\ge 3$}\label{subsection:m1m2m3ge3} 
In this subsection, we consider the subcase 
$\lambda=(m_1, m_2, m_3)$ with $m_1 \ge m_2 \ge m_3 \ge 3$. 
In this subcase, $n=m_1+m_2+m_3 \ge 9$.  

\begin{definition}\rm
For $i=3, 4, 5$, set 
\[
T_{1,2,i} = \left\{  \begin{array}{c|c} 
{\small\begin{ytableau}
        1 & a & \cdots \\
        2 & b & \cdots \\
        i & c & \cdots \\
\end{ytableau}} \in Y_{(m_1, m_2, m_3)} & {\small\begin{ytableau}
         a & \cdots \\
         b & \cdots \\
         c & \cdots 
\end{ytableau}} 
\begin{array}{c}
\mbox{ is not necessarily standard w.r.t. } \\ \{ 1, \ldots, n \}\setminus \{1, 2, i\}   
\end{array} 
\end{array} 
\right\} 
\]
and 
\[
T'_{1,2,i} = \left\{  \begin{array}{c|c} 
{\small\begin{ytableau}
        1 & a & \cdots \\
        2 & b & \cdots \\
        i & c & \cdots \\
\end{ytableau}} \in Y_{(m_1, m_2, m_3)} & {\small\begin{ytableau}
         a & \cdots \\
         b & \cdots \\
         c & \cdots 
\end{ytableau}} \mbox{ is standard w.r.t. } \{ 1, \ldots, n \}\setminus \{1, 2, i\}  
\end{array} 
\right\}. 
\]
(For the definition of a tableau being standard with respect to a subset $S \subseteq \{ 1, 2, \ldots, n\}$, see Definition~\ref{def:standardwrtS}.) 
Note that $T'_{1,2, 3} \coprod  T'_{1,2,4} \coprod  T'_{1,2,5} 
\subseteq Y^{\rm st}_{(m_1, m_2, m_3)}$ and 
that $\sharp T'_{1,2,3}= \sharp T'_{1,2,4} = \sharp T'_{1,2,5}= \dim_{{\Bbb K}} V_{(m_1-1, m_2-1, m_3-1)} = 
d(m_1-1, m_2-1, m_3-1)$.  
We put $Z_{(m_1, m_2, m_3)} = Y^{\rm st}_{(m_1, m_2, m_3)}\setminus (T'_{1,2,3} \coprod  T'_{1,2,4})$. 
\end{definition}

\begin{definition}\rm 
Let $W_1 = \langle T'_{1, 2, 3} \rangle = \langle T_{1, 2, 3}\rangle$. 
For $1\le i < j < k \le n$, choose a (not nec. standard) Young tableau $y_{i, j, k} \in Y_{(m_1, m_2, m_3)}$ whose 
first column is ${\small\begin{ytableau}
        i  \\
        j  \\
        k  \\
\end{ytableau}}$.  
We denote by $W_2$ the subspace of $V_{(m_1, m_2, m_3)}$ spanned by 
\begin{align*} 
T'_{1, 2, 3} & \cup \{ y_{1, i, j} \mid 4\le i < j \le n \} \cup \{ y_{2, i, j} \mid 4\le i < j \le n \}  
\cup \{ y_{3, i, j} \mid 4\le i < j \le n \} \\ 
& \cup \{ y_{1, 2, k} \mid 4\le k \le n \} \cup \{ y_{2, 3, k} \mid 4\le k \le n \}  
\cup \{ y_{1, 3, k} \mid 4\le k \le n \}. 
\end{align*} 
Then $\dim_{{\Bbb K}} W_1 =\sharp T'_{1, 2, 3}$ and $\dim_{{\Bbb K}} W_2 \le \sharp T'_{1, 2, 3} + G(n)$, 
where 
\[
G(n) = \dfrac{3(n-3)(n-4)}{2} + 3(n-3) = \dfrac{3(n-2)(n-3)}{2}.
\] 
\end{definition}

\begin{remark}\rm
We can verify that $\sigma(W_1) \cap W_2 \neq 0$ for any $\sigma \in S_n$.   
Indeed, 
\begin{align*}
\sigma(W_1) = \left\langle  \begin{array}{c|c} 
{\small\begin{ytableau}
\ytableausetup{boxsize=2em} 
        \sigma(1) & a & \cdots \\
        \sigma(2) & b & \cdots \\
        \sigma(3) & c & \cdots \\
\end{ytableau}} \quad 
& \quad {\small\begin{ytableau}
         a & \cdots \\
         b & \cdots \\
         c & \cdots 
\end{ytableau}} 
\begin{array}{c}
\mbox{ is not necessarily standard w.r.t. } \\ \{ 1, \ldots, n \}\setminus \{\sigma(1), \sigma(2), \sigma(3)\}   
\end{array} 
\end{array} \hspace*{-2ex} 
\right\rangle.    
\end{align*}
If $\{ \sigma(1), \sigma(2), \sigma(3) \} = \{ 1, 2, 3\}$, then $\sigma(W_1)=W_1 \subseteq W_2$ and hence $\sigma(W_1)\cap W_2= W_1\neq 0$. 
If $\{ \sigma(1), \sigma(2), \sigma(3) \} \cap \{ 1, 2, 3\} = \emptyset$, then
\begin{align*}
{\small\begin{ytableau}
\ytableausetup{boxsize=2em} 
        \sigma(1) & 1 & \cdots \\
        \sigma(2) & 2 & \cdots \\
        \sigma(3) & 3 & \cdots \\
\end{ytableau}} = 
{\small\begin{ytableau}
\ytableausetup{boxsize=2em} 
      1 &  \sigma(1)  & \cdots \\
      2 &  \sigma(2)  & \cdots \\
      3 &  \sigma(3)  & \cdots \\
\end{ytableau}} \in \sigma(W_1)\cap W_1 \subseteq \sigma(W_1) \cap W_2 \neq 0. 
\end{align*}
In the remaining cases, we can check that $\sigma(W_1) \cap W_2 \neq 0$, since 
\[
y_{1, i, j},\; y_{2, i, j},\; y_{3, i, j},\; y_{1, 2, k},\; y_{2, 3, k},\; y_{1, 3, k} \in W_2
\] 
for $4 \le i < j \le n$ and $4 \le k \le n$.  

\ytableausetup{boxsize=normal}

If we can prove that 
\[
d(m_1, m_2, m_3) \ge 2d(m_1-1, m_2-1, m_3-1) + G(n), 
\]
then we obtain $\dim_{{\Bbb K}} V_{(m_1, m_2, m_3)} \ge \dim_{{\Bbb K}} W_1+\dim_{{\Bbb K}} W_2$, and hence 
$V_{(m_1, m_2, m_3)}$ is not thick. 
\end{remark}

\bigskip 

To prove that $V_{(m_1, m_2, m_3)}$ is not thick, we only need to show that 
\begin{align}
\dfrac{d(m_1, m_2, m_3)}{d(m_1-1, m_2-1, m_3-1)} \ge 2 + \dfrac{G(n)}{d(m_1-1, m_2-1, m_3-1)}.  \label{eq:ineqofdepth3} 
\end{align} 
By (\ref{eq:dm1m2m3}) in Proposition~\ref{prop:formuladim}, 
\[
\dfrac{d(m_1, m_2, m_3)}{d(m_1-1, m_2-1, m_3-1)} = 
\dfrac{(m_1+m_2+m_3)(m_1+m_2+m_3-1)(m_1+m_2+m_3-2)}{(m_1+2)(m_2+1)m_3}. 
\]
For $m_1\ge m_2\ge m_3 \ge 3$, we have 
\begin{align*}
m_1+m_2+m_3-2 & \ge m_1+2, \\
m_1+m_2+m_3-1 & \ge 2(m_2+1), \\
m_1+m_2+m_3 & \ge 3m_3. 
\end{align*}
Hence, $\dfrac{d(m_1, m_2, m_3)}{d(m_1-1, m_2-1, m_3-1)} \ge 6$. 

On the other hand, let us consider $\dfrac{G(n)}{d(m_1-1, m_2-1, m_3-1)}$. 
By direct calculation, 
\begin{align}
& \dfrac{G(n)}{d(m_1-1, m_2-1, m_3-1)} \label{eq:Gnhm-1} \\ 
& =  \dfrac{3(n-2)(n-3)}{2}\cdot \dfrac{(m_1+1)!m_2!(m_3-1)!}{(m_1+m_2+m_3-3)!(m_1-m_2+1)(m_1-m_3+2)(m_2-m_3+1)} \nonumber \\
& = \dfrac{3(n-2)(n-3)}{2}\cdot \dfrac{m_1!(m_2-1)!(m_3-2)!}{(m_1+m_2+m_3-3)!}\cdot 
\dfrac{(m_1+1)m_2(m_3-1)}{(m_1-m_2+1)(m_1-m_3+2)(m_2-m_3+1)}.  \nonumber
\end{align}

\bigskip 

First, let us consider the case $m_3=3$. 
If $m_2\ge 4$, then 
\begin{align*}
& \dfrac{G(n)}{d(m_1-1, m_2-1, m_3-1)} \\
& = \dfrac{3(n-2)(n-3)}{2}\cdot \dfrac{m_1!(m_2-1)!}{(m_1+m_2)!}\cdot 
\dfrac{2(m_1+1)m_2}{(m_1-m_2+1)(m_1-1)(m_2-2)} \\
& = 3(m_1+m_2+1)(m_1+m_2)\cdot \dfrac{m_1!m_2!}{(m_1+m_2)!}\cdot 
\dfrac{(m_1+1)}{(m_1-m_2+1)(m_1-1)(m_2-2)} \\
& \le 3(2m_1+1)(2m_1)\cdot \dfrac{m_1!4!}{(m_1+4)!}\cdot 
\dfrac{(m_1+1)}{(m_1-m_2+1)(m_1-1)(m_2-2)} \\ 
& = 288\cdot\dfrac{(m_1+1/2)m_1}{(m_1+2)(m_1-1)}\cdot \dfrac{1}{(m_1-m_2+1)(m_1+4)(m_1+3)(m_2-2)} \\ 
& \le 288 \cdot \dfrac{1}{(m_1-m_2+1)(m_1+4)(m_1+3)(m_2-2)} \\ 
& \le 288 \cdot \dfrac{1}{1\cdot 8 \cdot 7\cdot 2} = \dfrac{18}{7} <4 
\end{align*} 
by Lemma~\ref{lemma:binomineq-1}, 
since $(m_1+1/2)m_1\le (m_1+2)(m_1-1)$ for $m_1\ge 4$. 
In this case, (\ref{eq:ineqofdepth3}) holds. 

When $m_2=3$, 
\begin{align*}
& \dfrac{G(n)}{d(m_1-1, m_2-1, m_3-1)} \\
& = 3(m_1+4)(m_1+3)\cdot \dfrac{m_1!3!}{(m_1+3)!}\cdot 
\dfrac{(m_1+1)}{(m_1-2)(m_1-1)} \\
& = \dfrac{18(m_1+4)}{(m_1+2)(m_1-2)(m_1-1)} \le 4
\end{align*} 
for any $m_1 \ge 4$. Then (\ref{eq:ineqofdepth3}) holds. 

The case $(m_1, m_2, m_3) = (3, 3, 3)$ remains. 
When $(m_1, m_2, m_3) = (3, 3, 3)$, we see that $d(3, 3, 3)=42$, $\dim_{{\Bbb K}} W_1=d(2, 2, 2)=5$. 
Let us redefine $W_2$. By the definition of $W_1$, we see that 
for any $i, j, k \in \{4, 5, \ldots, 9 \}$, $W_1$ contains a (not necessarily standard) 
Young tableau $y \in Y_{(3, 3, 3)}$ whose second column is $(i, j, k)$.  
We define $W_2$ to be the subspace of $V_{(3, 3, 3)}$ spanned by $W_1$ and the union of 
\begin{align*}
& \left\{ \begin{array}{c} 
{\small\begin{ytableau}
        a & b & c \\
        4 & 6 & 8 \\
        5 & 7 & 9 \\
\end{ytableau}}, \quad 
{\small\begin{ytableau}
        a & b & c \\
        4 & 5 & 7 \\
        6 & 8 & 9 \\
\end{ytableau}},  \quad 
{\small\begin{ytableau}
        a & b & c \\
        4 & 5 & 6 \\
        7 & 9 & 8 \\
\end{ytableau}}, \quad  
{\small\begin{ytableau}
        a & b & c \\
        4 & 5 & 6 \\
        8 & 7 & 9 \\
\end{ytableau}} 
 \end{array} \right\} \\ 
& \cup  \left\{ \begin{array}{c|c} 
{\small\begin{ytableau}
        a & b & \ast \\
        4 & c & \ast \\
        9 & i & \ast \\
\end{ytableau}} & i=5, 6, 7 \end{array} \right\} \cup 
\left\{ \begin{array}{c|c} 
{\small\begin{ytableau}
        a & b & \ast \\
        5 & c & \ast \\
        6 & i & \ast \\
\end{ytableau}} & i=8, 9 \end{array} \right\} \cup 
\left\{ \begin{array}{c|c} 
{\small\begin{ytableau}
        a & b & \ast \\
        7 & c & \ast \\
        8 & i & \ast \\
\end{ytableau}}  & i=4 \end{array} \right\},  
\end{align*}
where $(a, b, c)$ runs over $\{(1, 2, 3), (2, 3, 1), (3, 1, 2)\}$. 
Note that $\dim_{{\Bbb K}} W_2 \le \dim_{{\Bbb K}} W_1+4\cdot 3+6\cdot 3=35$ and that 
$\dim_{{\Bbb K}} W_1+\dim_{{\Bbb K}} W_2=40 \le 42= \dim_{{\Bbb K}} V_{(3, 3, 3)}$. 
We can verify that $\sigma(W_1)\cap W_2\neq 0$ for any $\sigma \in S_9$. 
Hence, $V_{(3, 3, 3)}$ is not thick. 
Thus, we have proved Theorem~\ref{th:main4} and Proposition~\ref{prop:main5} in 
the case $\lambda = (m_1, m_2, 3)$ with $m_1\ge m_2 \ge 3$.  

\bigskip 

Summarizing the discussion above, we have 

\begin{proposition}
In the case $\lambda=(m_1, m_2, 3)$ with $m_1\ge m_2\ge 3$, excluding 
$\lambda=(3, 3, 3)$, $V_{\lambda}$ is not $\left(d(m_1-1, m_2-1, 2), d(m_1-1, m_2-1, 2) + \dfrac{3(n-2)(n-3)}{2}\right)$-thick, where $n=m_1+m_2+3$. In particular, it is not $m$-thick for any $d(m_1-1, m_2-1, 2) \le m \le d(m_1, m_2, 3)-d(m_1-1, m_2-1, 2)-\dfrac{3(n-2)(n-3)}{2}$. 
\end{proposition} 

\begin{proposition}
For $\lambda=(3, 3, 3)$, $V_{\lambda}$ is not $(5, 35)$-thick. In particular, it is not $m$-thick for any $5 \le m \le 7$. 
\end{proposition}

\bigskip

Finally, let us consider the case $m_3\ge 4$. 
If $m_1\ge 5$, then 
\begin{align*}
\dfrac{(m_1+m_2+m_3-3)!}{m_1!(m_2-1)!(m_3-2)!} & = \binom{m_1+m_2+m_3-3}{m_1}
\cdot \binom{m_2+m_3-3}{m_2-1} \\
& \ge \binom{n-3}{5}
\cdot \binom{m_2+m_3-3}{2} \\  
& = \dfrac{(n-3)(n-4)(n-5)(n-6)(n-7)}{120} \cdot \dfrac{(m_2+m_3-3)(m_2+m_3-4)}{2}  
\end{align*}
by Lemma~\ref{lemma:binomineq-1}, 
since $m_2\ge m_3\ge 4$ and $n=m_1+m_2+m_3 \ge 13$. 
By (\ref{eq:Gnhm-1}), 
\begin{align*}
& \dfrac{G(n)}{d(m_1-1, m_2-1, m_3-1)}  \\ 
& \le \dfrac{3(n-2)(n-3)}{2}\cdot 
\dfrac{120}{(n-3)(n-4)(n-5)(n-6)(n-7)} \cdot \dfrac{2}{(m_2+m_3-3)(m_2+m_3-4)} \\
& \quad \cdot  \dfrac{(m_1+1)m_2(m_3-1)}{(m_1-m_2+1)(m_1-m_3+2)(m_2-m_3+1)} \\
& = 720 \cdot \dfrac{m_1+1}{n-7}\cdot \dfrac{m_2}{m_2+m_3-3}\cdot \dfrac{m_3-1}{m_2+m_3-4} \cdot \dfrac{n-2}{2(n-6)} \\
& \quad \cdot 
\dfrac{1}{(n-4)(n-5)(m_1-m_2+1)(m_1-m_3+2)(m_2-m_3+1)} \\
& \le 720\cdot 1 \cdot 1 \cdot \dfrac{m_3-1}{2m_3-4} \cdot \left( \dfrac{1}{2} + \dfrac{2}{n-6} \right) \cdot \dfrac{1}{9\cdot 8 \cdot 1\cdot 2 \cdot 1} \\ 
& \le 720 \cdot\dfrac{3}{4} \cdot \left(\dfrac{1}{2}+\dfrac{2}{13-6}\right)\cdot \dfrac{1}{9\cdot 8 \cdot 2} = \dfrac{165}{56} < 4. 
\end{align*}
Here, we used 
\begin{align*}
m_1+1 & \le n-7=m_1+m_2+m_3-7, \\
m_2 & \le m_2+m_3-3, \\ 
\text{~and~} \quad 
\dfrac{m_3-1}{m_2+m_3-4} & \le \dfrac{m_3-1}{2m_3-4} = \dfrac{1}{2}+\dfrac{1}{2m_3-4} \le \dfrac{1}{2}+\dfrac{1}{2\cdot 4-4} 
=\dfrac{3}{4} 
\end{align*} 
for $m_1 \ge m_2 \ge m_3 \ge 4$ and $n\ge 13$. 
Hence, we have verified the inequality (\ref{eq:ineqofdepth3}) except in the case 
$(m_1, m_2, m_3) = (4, 4, 4)$.

The case $(m_1, m_2, m_3) = (4, 4, 4)$ remains. 
When $(m_1, m_2, m_3) = (4, 4, 4)$, we see that $d(4, 4, 4)=462$, $\dim_{{\Bbb K}} W_1=d(3, 3, 3)=42$, and $G(12)=135$. The inequality (\ref{eq:ineqofdepth3}) holds because 
$d(4, 4, 4)=462 \ge 2d(3, 3, 3)+G(12) = 219$. 

Thus, we have proved Theorem~\ref{th:main4} and Proposition~\ref{prop:main5} in 
the case $\lambda = (m_1, m_2, m_3)$ with $m_1\ge m_2 \ge m_3 \ge 4$.  

\bigskip 

Summarizing the discussion above, we have 

\begin{proposition}
In the case $\lambda=(m_1, m_2, m_3)$ with $m_1\ge m_2\ge m_3 \ge 4$, excluding 
$\lambda=(4, 4, 4)$, $V_{\lambda}$ is not $(d(m_1-1, m_2-1, m_3-1), d(m_1-1, m_2-1, m_3-1)  + \dfrac{3(n-2)(n-3)}{2})$-thick, where $n=m_1+m_2+m_3$. In particular, it is not $m$-thick for any $d(m_1-1, m_2-1, m_3-1) \le m \le d(m_1, m_2, m_3)-d(m_1-1, m_2-1, m_3-1)-\dfrac{3(n-2)(n-3)}{2}$. 
\end{proposition} 

\begin{proposition}
For $\lambda=(4, 4, 4)$, $V_{\lambda}$ is not $(42, 177)$-thick. In particular, it is not $m$-thick for any $42 \le m \le 420$. 
\end{proposition} 

\proof 
We need only to prove the latter part. 
By $d(4, 4, 4)=462$ and Corollary~\ref{cor:ijthick}, we can verify the latter part. 
\qed 

\bigskip 

Therefore, Proposition~\ref{prop:main5} is proved in Case~2 by the discussion in Section~\ref{section:s=3}. 

%


\section{The case $\lambda = (m_1, m_2, \ldots, m_s)$ with $s \ge 4$}\label{section:sge4}  

In this section, we prove Proposition~\ref{prop:main5} in Case~3; namely, that if 
$\lambda=(m_1, \ldots, m_s)$ with $s\ge 4$ and  
$m_1\ge \cdots \ge m_s \ge 1$, excluding 
$\lambda=(1^n)$ and $(2, 1^{n-2})$, then $V_{\lambda}$ is not thick. 
For proving Theorem~\ref{th:main4} and Proposition~\ref{prop:main5} in Case~3, 
it suffices to show the case $s\ge 4$ and $m_1\ge 4$ by 
Proposition~\ref{prop:dualrep}, 
because we have proved the case $\lambda=(m_1, \ldots, m_s)$ with 
$s=2, 3$ in the previous sections.

\begin{definition}[{{\it cf.} \cite[page~4]{Fulton}}]\rm
Let two Young diagrams $\lambda=(\lambda_1, \lambda_2, \ldots)$ and $\mu = (\mu_1, \mu_2, \ldots)$ satisfy $\mu_i \le \lambda_i$ for all $i$ (that is, $\mu$ is contained in $\lambda$).   
A {\it skew diagram}, or {\it skew shape}, $\lambda/\mu$ is the diagram obtained by removing $\mu$ from $\lambda$. 
A {\it standard skew tableau} on $\lambda/\mu$ is a filling of the boxes of $\lambda/\mu$ 
with the numbers $\{ 1, 2, \ldots, |\lambda|-|\mu| \}$, each appearing exactly once, 
that is strictly increasing in rows and strictly increasing in columns. We denote by 
\[
S^{\rm st}_{\lambda/\mu} = \{ \mbox{standard skew tableau on $\lambda/\mu$} \} 
\]
the set of standard skew tableaux on $\lambda/\mu$. For example, if $\lambda=(4, 2, 2, 1)$ and 
$\mu = (1, 1)$, then 
\[
{\small\begin{ytableau}
       \none & 1 & 4 & 5 \\
       \none  & 3 \\
       2  & 6\\
      7  
\end{ytableau}}. 
\]
is a standard skew tableau on $\lambda/\mu$. 
\end{definition}

\begin{definition}\rm
Let $\lambda=(m_1, m_2, \ldots, m_s)$ with $s \ge 4$ and 
$m_1\ge 4$. 
Set 
\[ 
W_1 = \left\langle 
\begin{array}{c|c} f(y) & 
y = {\small\begin{ytableau}
        1 & \ast & \cdots   \\
        2 & \cdots \\
        \vdots  & \cdots  
\end{ytableau}}  \mbox{ is a standard $\lambda$-tableau }  
\end{array}  
\right\rangle
\]   
and  
\[ 
W'_1 = \left\langle 
\begin{array}{c|c} f(y) & 
y = {\small\begin{ytableau}
        1 & 2  & \cdots \\
        \ast &  \cdots  \\
        \vdots & \cdots 
\end{ytableau}}  \mbox{ is a standard $\lambda$-tableau }  
\end{array}  
\right\rangle.
\]   
Note that $V_{\lambda} = W_1\oplus W'_1$ and that $\dim_{{\Bbb K}} V_{\lambda} = \dim_{{\Bbb K}} W_1+\dim_{{\Bbb K}} W'_1$.  
Moreover, $\dim_{{\Bbb K}} W_1=\sharp S^{\rm st}_{\lambda/(1,1)}$ and $\dim_{{\Bbb K}} W'_1=\sharp S^{\rm st}_{\lambda/(2)}$. 
\end{definition}

\begin{proposition}
Proposition~\ref{prop:main5} holds for $\lambda=(m_1, m_2, \ldots, m_s)$ with 
$s\ge 4$ and $m_1\ge 4$. In particular, $V_{\lambda}$ is not thick in this case. 
\end{proposition}

\proof 
To prove the case $\lambda=(m_1, m_2, \ldots, m_s)$ with 
$s\ge 4$ and $m_1\ge 4$, it suffices to show that 
either $V_{\lambda}$ or $V_{{}^t \lambda}$ is not thick by Proposition~\ref{prop:dualrep}.  
If $\dim_{{\Bbb K}} W'_1 \le \dim_{{\Bbb K}} W_1$, then change $V_{\lambda}$ into $V_{{}^t \lambda}$. 
Since $\sharp S^{\rm st}_{\lambda/(1,1)}=\sharp S^{\rm st}_{{}^t \lambda/(2)}$ and 
$\sharp S^{\rm st}_{\lambda/(2)} = \sharp S^{\rm st}_{{}^t \lambda/(1,1)}$, we may assume that 
$\dim_{{\Bbb K}} W_1 \le \dim_{{\Bbb K}} W'_1$ from the beginning. Then we have $2 \dim_{{\Bbb K}} W_1 \le \dim_{{\Bbb K}} V_{\lambda}$. 

By the Garnir relation (Theorem~\ref{th:Garnir's relation}), we see that 
\[
W_1 = \left\langle 
\begin{array}{c|c} f(y) & 
y = {\small\begin{ytableau}
        1 & \ast & \cdots   \\
        2 & \cdots \\
        \vdots  & \cdots  
\end{ytableau}}  \mbox{ is a (not necessarily standard) $\lambda$-tableau }  
\end{array}  
\right\rangle.   
\]
In particular, for any $3 \le i < j \le n$, there exists a $\lambda$-tableau $y$ 
whose first column contains $\{1, 2, i, j\}$ such that $f(y) \in W_1$.  
For $\sigma \in S_n$,  
\[
\ytableausetup{boxsize=2.2em}
\sigma(W_1) = \left\langle 
\begin{array}{c|c} f(y) & 
y = {\small\begin{ytableau}
        \sigma(1) & \ast & \cdots   \\
        \sigma(2) & \cdots \\
        \vdots  & \cdots  
\end{ytableau}}  \mbox{ is a (not nec. standard) $\lambda$-tableau }  
\end{array}  
\right\rangle.   
\]
We easily verify that there exists a $\lambda$-tableau $y$ whose first 
column contains $\{ \sigma(1), \sigma(2) \}$ such that $f(y) \in W_1$. 
Hence, $\sigma(W_1)\cap W_1\neq 0$. 

Setting $W_2=W_1$, we see that $\sigma(W_1)\cap W_2\neq 0$ for any $\sigma \in S_n$ and 
$\dim_{{\Bbb K}} W_1+\dim_{{\Bbb K}} W_2\le \dim_{{\Bbb K}} V_{\lambda}$, which implies that 
Proposition~\ref{prop:main5} holds in 
the case $\lambda=(m_1, m_2, \ldots, m_s)$ with $s\ge 4$ and $m_1\ge 4$.  
\qed

\bigskip 

Summarizing the discussion above, we have 

\begin{proposition}
In the case  $\lambda=(m_1, m_2, \ldots, m_s)$ with 
$s\ge 4$ and $m_1\ge 4$, set $d_0=\min\{ \sharp S^{\rm st}_{\lambda/(1,1)}, \sharp S^{\rm st}_{\lambda/(2)}\}$. Then 
$V_{\lambda}$ is not $(d_0, d_0)$-thick. In particular, it is not $m$-thick for 
any $d_0 \le m \le d(\lambda)-d_0$. 
\end{proposition}

\proof 
The latter part follows from Proposition~\ref{prop:ijthick}.  
\qed 

\bigskip

Thus, Proposition~\ref{prop:main5} has been proved in Case~3. 
By the arguments in Sections~\ref{section:s=2}--\ref{section:sge4}, we have proved 
Theorem~\ref{th:main4} and Proposition~\ref{prop:main5}, which implies Theorems~\ref{th:main1} and \ref{th:main2}.

\end{document}